\documentclass[12pt]{article}
\usepackage[margin=1in]{geometry} %%%%Margin
\usepackage{authblk} %%Author
\usepackage{url} %%%Url
\usepackage{hyperref}  
\usepackage{wrapfig,xcolor}

\definecolor{darkred}{rgb}{0.7, 0.0, 0.0}
\hypersetup{
    colorlinks=true,
    urlcolor = {darkred},
    linkcolor = {darkred},
    citecolor = {darkred},
    linkcolor = {darkred}
}

\usepackage{amsthm,amssymb,amsmath}
\usepackage{enumitem}
\usepackage{tikz}
\usepackage{tikz-cd}
\usepackage{dsfont}
\usepackage{xspace} 
\usepackage[normalem]{ulem}
\usepackage{lipsum}
\usepackage{amsfonts}
\usepackage{graphicx}
\usepackage{epstopdf}
\usepackage{algorithmic}
\ifpdf
  \DeclareGraphicsExtensions{.eps,.pdf,.png,.jpg}
\else
  \DeclareGraphicsExtensions{.eps}
\fi

\usepackage{enumitem}
\setlist[enumerate]{leftmargin=.5in}
\setlist[itemize]{leftmargin=.5in}

\newtheorem{theorem}{Theorem}[section]

\newtheorem{proposition}[theorem]{Proposition}
\newtheorem{remark}[theorem]{Remark}
\newtheorem{definition}[theorem]{Definition}
\newtheorem{lemma}[theorem]{Lemma}
\newtheorem{example}[theorem]{Example}
\newtheorem{corollary}[theorem]{Corollary}
\usepackage[ruled,vlined, linesnumbered]{algorithm2e}
\newcommand{\aug}{\mathbf{Aug}}
\newcommand{\ult}{\mathbf{Ult}}
\newcommand{\rep}{\mathbf{Rep}(e,m)}
\newcommand{\stair}{\mathbf{Dec}}
\newcommand{\rec}{\mathbf{Rec}}
\newcommand{\eld}{\mathbf{ER}}
\newcommand{\rank}{\mathrm{rank}}

\newcommand \XX{\mathcal{X}=(X,d_X,f_X)}

\newcommand\supp{\mathrm{supp}}

\newcommand\R{\mathbb{R}}
\newcommand\N{\mathbb{N}}
\newcommand\Z{\mathbb{Z}}
\newcommand\ZP{\mathbb{Z}_{\geq0}}
\newcommand\RP{\mathbb{R}_{\geq0}}
\newcommand\A{\mathcal{A}}
\newcommand\B{\mathcal{B}}

\newcommand\Ha{\mathrm{H}}
\newcommand{\eps}{\varepsilon}
\newcommand \rips{\mathcal{R}}
\newcommand \ripsbi{\mathcal{R}^{\mathrm{bi}}}
\newcommand \Pb{\mathbb{P}}
\newcommand \F{\mathbb{F}}
\newcommand \free{\mathcal{F}_{\F}}

\newcommand \hzero{\mathrm{H}_0} 
\newcommand \Hrm{\mathrm{H}}

\newcommand \im{\mathrm{im}}
\newcommand{\vect}{\mathbf{Vec}}

\newcommand{\coker}{\mathrm{coker}}
\newcommand{\hf}{\mathrm{dm}}

\newcommand{\simp}{\mathbf{Simp}}
\newcommand{\gr}{\mathrm{gr}}
\newcommand{\barc}{\mathbf{barc}}
\newcommand{\subpart}{\mathbf{Subpart}}

\newcommand{\Lcal}{\mathcal{L}}

\newcommand{\ba}{\mathbf{a}}
\newcommand{\bb}{\mathbf{b}}
\newcommand{\bc}{\mathbf{c}}
\newcommand{\bx}{\mathbf{x}}
\newcommand{\by}{\mathbf{y}}

\newcommand{\bn}{\mathbf{n}}
\newcommand{\bv}{\mathbf{v}}
\newcommand{\K}{\mathcal{K}}
\newcommand{\Scal}{\mathcal{S}}
\newcommand{\barcp}{\mathbf{barc}}
\newcommand{\one}{\mathds{1}}

\newcommand \X{\mathcal{X}}

\newcommand \Y{\mathcal{Y}}

\newcommand \I {\mathcal{I}}

\newcommand{\lmulti}{\left\{\!\!\left\{}
\newcommand{\rmulti}{\right\}\!\!\right\}}
\newcommand{\abs}[1]{\left\lvert#1\right\rvert}

\newcommand{\filterMS}  {augmented metric space}
\newcommand{\newbarcode}  {elder-rule-staircode} 
\newcommand{\Newbarcode}  {Elder-rule-staircode}
\newcommand{\newstair}   {staircode}
\newcommand{\Newstair}   {Staircode}

\newcommand{\afilterMS}  {an augmented metric space}

\newcommand{\anewbarcode}  {an elder-rule-staircode}
\newcommand{\newbarcodes}  {elder-rule-staircodes} 

\newcommand{\Newbarcodes}  {Elder-rule-staircodes}

\newcommand{\elder}{elder-rule-barcode}
 
\newcommand{\sigeps}{(\sigma,\eps)}

\newcommand{\bitree}{\theta_X^{\mathrm{bi}}}
\newcommand{\bitreex}{\theta_\X^{\mathrm{bi}}}

\newcommand{\BML}  {|B^L|}

\newcommand{\afilterMSshort}  {an aug-MS} 

\newcommand{\anewbarcodeshort}  {an ER-staircode}
\newcommand{\newbarcodesshort}  {ER-staircodes}

\newcommand{\filterMSsshort}  {aug-MSs}

\newcommand{\filterMSshort}  {aug-MS}
\newcommand{\newbarcodeshort}  {ER-staircode}

\title{Elder-Rule-Staircodes for Augmented Metric Spaces}

\author[1]{Chen Cai\thanks{\texttt{cai.507@osu.edu}}}
\author[3]{Woojin Kim\thanks{\texttt{woojin.kim205@duke.edu}}}
\author[1,2]{Facundo M\'emoli\thanks{\texttt{memoli@math.osu.edu}}}
\author[1]{Yusu Wang\thanks{\texttt{yusu@cse.ohio-state.edu}}}
\affil[1]{Department of Computer Science and Engineering, 
		The Ohio State University.}
\affil[2]{Department of Mathematics, 
		The Ohio State University.
		}
\affil[3]{Department of Mathematics, 
		Duke University.
		}

\begin{document}

\maketitle

\begin{abstract}
An \filterMS{} is a metric space $(X, d_X)$ equipped with a function $f_X: X \to \mathbb{R}$. This type of data arises commonly in practice, e.g, a point cloud $X$ in $\mathbb{R}^d$ where each point $x\in X$ has a density function value $f_X(x)$ associated to it. An \filterMS{} $(X, d_X, f_X)$ naturally gives rise to a 2-parameter filtration $\K$. However, the resulting 2-parameter persistent homology $\Hrm_{\bullet}(\K)$ could still be of wild representation type, and may not have simple indecomposables. In this paper, motivated by the elder-rule for the zeroth homology of 1-parameter filtration, we propose a barcode-like summary, called the \emph{\newbarcode{}}, as a way to encode $\hzero(\K)$. Specifically, if $n = |X|$, the \newbarcode{} consists of $n$ number of staircase-like blocks in the plane. We show that if $\hzero(\K)$ is interval decomposable, then the barcode of $\hzero(\K)$ is equal to the \newbarcode{}. Furthermore, regardless of the interval decomposability, the fibered barcode, the dimension function (a.k.a. the Hilbert function), and the graded Betti numbers of $\hzero(\K)$ can all be efficiently computed once the \newbarcode{} is given.
 Finally, we develop and implement an efficient algorithm to compute the \newbarcode{} in $O(n^2\log n)$ time, which can be improved to $O(n^2\alpha(n))$ if $X$ is from a fixed dimensional Euclidean space $\mathbb{R}^d$, where $\alpha(n)$ is the inverse Ackermann function. 
\end{abstract}

\section{Introduction}
\label{ref:introduction}

 An \filterMS{} is a metric space $(X, d_X)$ equipped with a function $f_X: X \to \mathbb{R}$ \cite{bauer2019cotorsion,carlsson2010multiparameter,chazal2011scalar}. This type of data arises commonly in practice: e.g, a point cloud $X$ in $\mathbb{R}^d$ where each point has a density function value $f_X$ associated to it.  Studying hierarchical clustering methods induced in this setting has attracted much attention starting with \cite{carlsson2010multiparameter} and more recently with \cite{bauer2019cotorsion,campello2013density,martinez2018density}.
 Another example is where $X=V$ equals to the vertex set of a  graph $G=(V, E)$, $d_X$ represents certain graph-induced metric on $X$ (e.g, the diffusion distance induced by $G$), and $f_X$ is some descriptor function (e.g, discrete Ricci curvature) at graph nodes. This graph setting  occurs often in practice for graph analysis applications, where $G$ can be viewed as a skeleton of a hidden domain. When summarizing or characterizing $G$, one wishes to take into consideration both the metric structure of this domain and node attributes. Given that persistence-based summaries from only the edge weights or from only node attributes have already shown promise in graph classification (e.g, \cite{cai2020understanding,carriere2019general,Hofer2017Deep,ZW19}), it would be highly desirable to incorporate (potentially more informative) summaries encoding both types of information to tackle such tasks. In brief, we wish to develop topological invariants induced from such \filterMS{}s. 

On the other hand, an \filterMS{} naturally gives rise to a 2-parameter filtration (by filtering both via $f_X$ and via distance $d_X$; see Definition \ref{def:Rips bifiltration}). However, while a standard (1-parameter) filtration and its induced persistence module has persistence diagram as a complete discrete invariant, multi-parameter persistence modules do not have such complete discrete invariant \cite{CZ09,cohen2007stability}. The 2-parameter persistence module induced from an \filterMS{} may still be of wild representation type, and may not have simple indecomposables \cite{bauer2019cotorsion}. Instead, several recent works consider informative (but not necessarily complete) invariants for multiparameter persistence modules \cite{dey2019generalized,harrington2019stratifying,kim2018rank, lesnick2015interactive,mccleary2019multiparameter,miller2017data,vipond2020multiparameter}. 
In particular, RIVET \cite{lesnick2015interactive} provides an interactive visualization of  barcodes associated to 1-dimensional slices of an input 2-parameter persistence module $M$, which are called the \emph{fibered barcode}. {For implementing the interactive aspect, RIVET  makes efficient use of  \emph{graded Betti numbers} of $M$, another invariant of the 2-parameter persistence module $M$.}

\paragraph{Our contributions.} 
We propose a barcode-like summary, called the \emph{\newbarcode{}}, as a way to encode the zeroth homology of the 2-parameter filtration induced by a finite \filterMS{}.  Specifically, given a finite $\XX$, its \newbarcode{} consists of $n = |X|$ number of staircase-like blocks of $O(n)$ descriptive complexity in the plane. The development of the \newbarcode{} is motivated by the elder-rule behind the construction of persistence pairing for a 1-parameter filtration \cite{edelsbrunner2010computational}. For the 1-parameter case, \emph{barcodes} \cite{zomorodian2005computing} can be obtained by the decomposition of persistence modules in the realm of commutative algebra, or equivalently, by applying the elder-rule which is flavored with combinatorics or order theory. As we describe in Section \ref{sec:decorated staircodes and treegrams}, our \newbarcode{}s are obtained by adapting the elder-rule for treegrams arisen from 1-parameter filtration.

Interestingly, we show that our \newbarcode{} encodes much of topological information of the 2-parameter filtration $\K$ induced by $\X$. In particular, the fibered barcodes, the fibered treegrams, and the graded Betti numbers associated to $\hzero(\K)$ can all be efficiently computed from the \newbarcode{}s (see Theorems \ref{thm:fibered barcodes}, \ref{thm:fibered treegrams} and \ref{thm:recover graded Bettis}).
{Furthermore, if $\hzero(\K)$ is interval decomposable, then the interval indecomposables appearing in its decomposition correspond exactly to its \newstair{} (see Theorem \ref{thm:strong elder rule}). This implies that testing the interval decomposability of $\hzero(\K)$ is reduced to testing isomorphism of two given persistence modules \cite{brooksbank2008testing} (see Remark \ref{rem:testing interval decomposability}).  We also provide sufficient conditions on $\X$ which ensure the interval decomposability of $\hzero(\K)$ (see Theorem \ref{thm:elder rule2} and Corollary \ref{cor:ultrametric}). Therefore, to explore exotic isomorphism types of indecomposable summands of $\hzero(\K)$ (a question of interest considered in \cite{bauer2019cotorsion}), it suffices to restrict our attention to augmented metric spaces which do not satisfy these conditions.}

Finally, in Section \ref{sec:algorithm}, we show that the \newbarcode{} can be computed in $O(n^2 \log n)$ time for a finite \filterMS{} $(X, d_X, f_X)$ where $n = |X|$, and $O(n^2 \alpha(n))$ time if $X$ is from a fixed dimensional Euclidean space and $d_X$ is Euclidean distance. 
We have software to compute \newbarcode{}s and to explore / retrieve information such as fibered barcodes interactively, which is available at \url{https://github.com/Chen-Cai-OSU/ER-staircode}.

\paragraph{More on related work.} 
The \emph{elder-rule} is an underlying principle for extracting the persistence diagram from a persistence module induced by a nested family of simplicial complexes \cite[Chapter 7]{edelsbrunner2010computational}. Recently this principle has come into the spotlight again for generalizing persistence diagrams \cite{kim2018rank,mccleary2019multiparameter,patel2018generalized} and for addressing inverse problems in TDA \cite{curry2018fiber}. {An algorithm for testing interval decomposability of multiparameter persistence modules has been studied  \cite{asashiba2018interval}. A method to approximate 2-parameter persistence modules by interval-decomposable persistence modules has been proposed  \cite{asashiba2019approximation}.}

The software RIVET and work of \cite{lesnick2019computing} can also be used to recover fibered barcodes and graded Betti numbers. However, for the special case of zeroth 2-parameter persistence modules induced from \filterMS{}s, our \newbarcode{}s are simpler and more efficient to achieve these goals: In particular, given an \filterMS{} containing $n$ points,  the algorithm of \cite{lesnick2019computing} computes the graded Betti numbers in $\Omega(n^3)$ time, while it takes $O(n^2 \log n)$ time using \newbarcode{} via Theorem \ref{thm:computation}. 
For zeroth fibered barcodes, RIVET takes $O(n^8)$ time to compute a data structure of size  $O(n^{6})$ so as to support efficient query time of $O(\log n + |B^L |)$ where $|B^L|$ is the size of the fibered barcode $B^L$ for a particular line $L$ of positive slope.
Our algorithm computes  \newbarcode{} of size $O(n^2)$ in $O(n^2\log n)$ time, after which $B^L$ can be computed in $O(|B^L| \log n)$ time for any query line $L$. See Section \ref{appendix:comparasion} for more detailed comparison. 
However, it is important to note that RIVET allows much broader inputs and can work beyond zeroth homology.

\paragraph{Outline.} {In Section \ref{sec:preliminaries} we review the definitions of persistence modules, barcodes, and graded Betti numbers.  In Section \ref{sec:newbarcode} we introduce a 2-parameter filtration $\K$ induced by an \filterMS{} $\X$ and define the \newbarcode{} of $\X$. In Section \ref{sec:decorated staircodes and treegrams} we show that the \newbarcode{} recovers the fibered barcode of $\hzero(\K)$. We also prove that, if $\hzero(\K)$ is interval decomposable, then the set of indecomposables corresponds exactly to the \newstair{}.  In Section \ref{sec:staircodes and Bettis} we show that the \newbarcode{} recovers the graded Betti numbers of $\hzero(\K)$. In Section \ref{sec:algorithm} we develop and implement an efficient algorithm to compute the \newbarcode{}. In Section \ref{sec:conclusion} we discuss open problems. For readability we have relegated some proofs to an appendix.}

\paragraph{Acknowledgements.} {The authors thank to the anonymous reviewers who made a number of helpful comments to improve the paper. Also, CC and WK thank Cheng Xin for helpful discussions. This work is supported by NSF grants DMS-1723003, CCF-1740761, DMS-1547357, and IIS-1815697.}

%%%%%%%%%%%%%%%%%%%%%%%%%%%%%%%%%%%%%%%%%%%%%%%%%%%%%%%%%%%%%%%%%%%%%%%%%%%%%%%%%%%%%%%%%%%%%%%%%%%%%%%%%%
\section{Preliminaries}\label{sec:preliminaries}
In Section \ref{sec:persistence modules I} we review the definitions of persistence modules and their barcodes. In Section \ref{sec:graded Betti numbers} we review the notion of graded Betti number of a persistence module.
\subsection{Persistence modules and their decompositions}\label{sec:persistence modules I}

First we briefly review the definition of persistence modules. Let $\Pb$ be a poset. We regard $\Pb$ as the category that has elements of $\Pb$ as objects. Also, for any $\ba,\bb\in \Pb$, there exists a unique morphism $\ba\rightarrow \bb$ if and only if $\ba\leq \bb$. For $d\in\N$, let $\Z^d$ be the set of $d$-tuples of integers equipped with the partial order defined as $(a_1,a_2,\ldots,a_d)\leq (b_1,b_2,\ldots,b_d)$ if and only if $a_i\leq b_i$ for each $i=1,2,\ldots,d$.  The poset structure on $\R^d$ is defined in the same way. 

We fix a certain field $\F$ and every vector space in this paper is over $\F$. Let $\vect$ denote the category of \emph{finite dimensional} vector spaces over $\F$.

A ($\Pb$-indexed) \emph{persistence module} is a functor $M:\Pb\rightarrow \vect$. In other words, to each $\ba\in \Pb$, a vector space $M(\ba)$ is associated, and to each pair $\ba\leq \bb$ in $\Pb$, a linear map $\varphi_{M}(\ba,\bb):M(\ba)\rightarrow M(\bb)$ is associated. When $\Pb=\R^d$ or $\Z^d$, $M$ is said to be a \emph{$d$-parameter persistence module}. A \emph{morphism} between  $M,N:\Pb\rightarrow \vect$ is a natural transformation $f:M\rightarrow N$ between $M$ and $N$. That is, $f$ is a collection $\{f_{\ba}\}_{\ba\in \Pb}$ of linear maps such that for every pair $\ba\leq\bb$ in $\Pb$, the following diagram commutes:
 \[\begin{tikzcd}
M(\ba) \arrow{r}{\varphi_M(\ba,\bb)} \arrow{d}{f_{\ba}}
&M(\bb) \arrow{d}{f_{\bb}}\\
N(\ba) \arrow{r}{\varphi_N(\ba,\bb)} &N(\bb).
\end{tikzcd}\]
Two persistence modules $M$ and $N$ are \emph{isomorphic}, denoted by $M\cong N$, if there exists a natural transformation $\{f_{\ba}\}_{\ba\in \Pb}$ from $M$ to $N$ where each $f_{\ba}$ is an isomorphism. 

We now review the standard definition of barcodes, following notation from \cite{botnan2018algebraic}. 
\begin{definition}[Intervals]\label{def:intervals}
Let $\Pb$ be a poset. An \emph{interval} $\mathcal{J}$ of $\Pb$ is a subset $\mathcal{J}\subset \Pb$ s.t.: 	(1) $\mathcal{J}$ is non-empty. 
 		(2) If $\ba,\bb\in \mathcal{J}$ and $\ba\leq \bc\leq \bb$, then $\bc\in \mathcal{J}$. 
 		(3) For any $\ba,\bb\in \mathcal{J}$, there is a sequence $\ba=\ba_0,
		\ba_1,\cdots,\ba_l=\bb$ of elements of $\mathcal{J}$ with $\ba_i$ and $\ba_{i+1}$ comparable for $0\leq i\leq l-1$.
\end{definition}

\label{interval module}For $\mathcal{J}$ an interval of $\Pb$, the \emph{interval module} ${I}^{\mathcal{J}}:\Pb\rightarrow \vect$ is defined as
\[{I}^{\mathcal{J}}(\ba)=\begin{cases}
\mathbb{F}&\mbox{if}\ \ba\in \mathcal{J},\\0
&\mbox{otherwise.} 
\end{cases}\hspace{20mm} \varphi_{I^{\mathcal{J}}}(\ba,\bb)=\begin{cases} \mathrm{id}_\mathbb{F}& \mbox{if} \,\,\ba,\bb\in\mathcal{J},\ \ba\leq \bb,\\ 0&\mbox{otherwise.}\end{cases}\]

Recall that \textit{a multiset} is a collection of objects (called elements) in which elements may occur more than once, and the number of instances of an element is its \textit{multiplicity}. 

\begin{definition}[Interval decomposability and barcodes]\label{def:interval decomposable}\label{def:barcode}A functor $M:\Pb\rightarrow \vect$ is \emph{interval decomposable} if there exists a multiset $\barcp(M)$ of intervals (Definition \ref{def:intervals}) of $\Pb$ such that $M\cong \bigoplus_{\mathcal{J}\in\barcp(M)}I^{\mathcal{J}}.$ We call $\barcp(M)$ the \emph{barcode} of $M.$
\end{definition}

By the theorem of Azumaya-Krull-Remak-Schmidt \cite{azumaya1950corrections}, such a decomposition is unique up to a permutation of the terms in the direct sum. Therefore, the multiset $\barcp(M)$ is unique if $M$ is interval decomposable. 
For $d=1$, any  $M:\R^d\ \mbox{(or $\Z^d$)}\ \rightarrow\vect$ is interval decomposable and thus $\barcp(M)$ exists. However, for $d\geq 2$, $M$ may not be interval decomposable. 

\subsection{Graded Betti numbers}\label{sec:graded Betti numbers}

\paragraph{Persistence module as a module over a polynomial ring.}  In Section \ref{sec:persistence modules I} we defined $d$-parameter persistence modules as $\vect$-valued functors over the posets $\Z^d$ or $\R^d$, and morphisms between them as natural transformations. Definitions below are equivalent to those definitions \cite[Theorem 1]{CZ09}, and allow us to define the graded Betti numbers of persistence modules.  We mostly adopt notation in \cite{dey2019generalized,lesnick2019computing}.

Let $\F[t_1,t_2,\ldots,t_d]$ be the polynomial ring in the $d$-variables $t_1,t_2,\ldots,t_d$.  To ease notation, for $\bn:=(n_1,n_2,\ldots,n_d)\in \ZP^d$, the monomial $t_1^{n_1}t_2^{n_2}\ldots t_d^{n_d}\in \F[t_1,t_2,\ldots,t_n]$ will be written as $\bx^\bn$. 
A $d$-parameter persistence module $M:\Z^d\rightarrow \vect$ is an $\F[t_1,t_2,\ldots,t_d]$-module $M$ with a direct sum decomposition as an $\F$-vector space $M\cong \bigoplus_{\ba\in \Z^d}M_{\ba}$ such that the action of $\F[t_1,t_2,\ldots,t_d]$ on $M$ is uniquely specified as follows: for all $\ba=(a_1,a_2,\ldots,a_d)\in \Z^d$ and $v\in M_{\ba}$, and  for all $\bn=(n_1,n_2,\ldots,n_d)\in \Z_{\geq0}^d$, and for all $c\in k$,
\[(c\cdot \bx^\bn)\cdot v:=c\cdot \varphi_M(\ba,\ba+\bn)(v).\]
Let $M$ and $N$ be any two persistence modules. A morphism $f:M\rightarrow N$ is the module homomorphism such that $f(M_{\ba})\subseteq N_{\ba}$ for all $\ba\in \Z^d$. 
Since our interests are in studying finite \filterMSsshort{}, we will restrict ourselves to \emph{finite persistence modules} \cite{CZ09}---the $k$-th homology of a filtration of a finite simplicial complex for some $k\in \ZP$--- in what follows.

The kernel, image, and cokernel of $f$ are analogously defined to those of a linear map between vector spaces. The \emph{kernel} of $f$ is defined as the submodule $\ker(f):=\bigoplus_{\ba\in \Z^d}\ker(f_{\ba})$ of $M$. The image of $f$ is defined as the submodule $\im(f):=\bigoplus_{\ba\in \Z^d}\im(f_{\ba})$ of $N$.  The \emph{cokernel} of $f$ is defined as $\coker(f):=\bigoplus_{\ba\in \Z^d}\left(N_{\ba}/\im(f_{\ba})\right)$.

\paragraph{Graded Betti numbers.} We briefly review the concept of \emph{graded Betti numbers} \cite{CZ09,knudson2007refinement,lesnick2015interactive,lesnick2019computing,peeva2010graded,zomorodian2005computing}. Since our interests are in studying finite \filterMSsshort{}, we restrict ourselves to \emph{finite} persistence modules ---the $k$-th homology of a filtration of a finite simplicial complex for some $k\in \ZP$ \cite{CZ09}.

Fix $\ba\in \Z^d$.  By $Q^{\ba}:\Z^d\rightarrow \vect$, we denote the persistence module defined as
\[ Q^{\ba}_\bx=\begin{cases} \F, &\mbox{if $\ba\leq \bx$}\\ 0, &\mbox{otherwise,}
\end{cases}\hspace{10mm} \varphi_{Q^{\ba}}(\bx,\by)=\begin{cases}\mathrm{id}_{\F},&\mbox{if $\ba\leq \bx$}\\0,&  \mbox{otherwise}.
\end{cases}
\]
Any $F:\Z^d\rightarrow \vect$ is said to be \emph{free} if there exists a multiset $\A$ of elements of $\Z^2$ such that $F\cong \bigoplus_{\ba\in \A}Q^{\ba}$. For simplicity, we will refer to free persistence modules as free modules.
Let $M$ be a persistence module. An element $m\in M_{\ba}$ for some $\ba\in \Z^d$ is called a homogeneous element of $M$. In this case, we write $\gr(m)=\ba$. 
Let $F$ be a free module. A \emph{basis} $B$ of $F$ is defined as a minimal homogeneous set of generators of $F$. 
There can exist two bases $B$ and $B'$ of $F$ (analogous to the fact that a vector space can have multiple bases). However, the number of elements at each grade $\ba\in \Z^d$ in a basis of $F$ is an isomorphism invariant.

For a finite $M$, let $IM$ denote the submodule of $M$  generated by the images of all linear maps $\varphi_{M}(\ba,\bb)$, with $\ba<\bb$ in $\Z^2$.
Assume that there is a chain of modules
\begin{equation}\label{eq:resolution}F^{\bullet}:\begin{tikzcd}\cdots\arrow{r}{\partial_3}&F^2\arrow{r}{\partial_2}&F^1\arrow{r}{\partial_1}&F^0 \arrow{r}{\partial_0} 
&M \arrow{r}{0(=:\partial_{-1})}& 0\end{tikzcd}
\end{equation}
such that (1) each $F^i$ is a free module, and (2) $\im(\partial^i)=\ker(\partial^{i-1})$, $i=0,1,2,\cdots.$ Then we call $F^{\bullet}$ a \emph{resolution} of $M$. The condition (2) is referred to as \emph{exactness} of $F^{\bullet}$. We call the resolution $F^{\bullet}$ \emph{minimal} if $\im(\partial^i) \subseteq IF^{i-1}$ for $i=1,2,\cdots$. It is a standard fact that a minimal resolution of $M$ always exists and is unique up to isomorphism \cite[Chapter I]{peeva2010graded}.

\begin{definition}[Graded Betti numbers]\label{def:graded betti} Let $M:\Z^d\rightarrow \vect$ be finite. 
Assume that a minimal free resolution of $M$ is $F^{\bullet}$ in (\ref{eq:resolution}). 
For $i\in \ZP$, the $i$-th graded Betti number $\beta_i^M:\Z^d\rightarrow \ZP$ is defined as $\beta_i^M(\ba)=\mbox{(number of elements at grade $\ba$ in any basis of $F^i$)}.$
\end{definition}

\begin{remark}\label{rem:graded betti numbers}
\begin{enumerate}[label=(\roman*)]
    \item {Note that if $M\cong N_1\bigoplus N_2$, then $\beta_{i}^M=\beta_{i}^{N_1}+\beta_{i}^{N_2}$. This is a key fact to define the \emph{persistent graded Betti numbers} introduced in \cite{dey2019generalized}.\label{item:graded betti numbers1}}
    \item 
    $\beta_i^M:\Z^d\rightarrow \ZP$ is the zero function  for every integer $i> d$ {\cite[Theorem 1.13]{eisenbud2013commutative}}.
    \item   Definition \ref{def:graded betti} is not in the exactly same form as those in the literature such as \cite{CZ09,knudson2007refinement,lesnick2015interactive}. However, by Nakayama's lemma \cite[Lemma 2.11]{peeva2010graded} all those are equivalent, as already noted in \cite[Section 2.3]{lesnick2019computing}.
\end{enumerate}

\end{remark}

For any $M:\Z^d\rightarrow \vect$, the \emph{dimension function} $\hf(M):\Z^d\rightarrow \ZP$ of $M$ is defined as $\ba\mapsto \dim M_{\ba}$. The graded Betti numbers of $M$ recover $\hf(M)$:
\begin{theorem}[{\cite[Proposition 2.3]{lesnick2019computing}}]\label{thm:betti formula}Let $M:\Z^d\rightarrow \vect$ be a finite persistence module. For all $\ba\in \Z^d$,
 \[\hf(M)(\ba)=\sum_{\bx\leq \ba}\sum_{i=0}^d(-1)^i\beta_i^M(\bx).
 \]
\end{theorem}

%%%%%%%%%%%%%%%%%%%%%%%%%%%%%%%%%%%%%%%%%%%%%%%%%%%%%%%%%%%%%%%%%%%%%%%%%%%%%%%%%%%%%%%%%
\section{{\Newbarcode{}s} for \filterMS{}s }\label{sec:newbarcode}

\paragraph{Rips bifiltration for \afilterMSshort.}  Let $(X,d_X)$ be a metric space. For $\eps\in \R$, the \emph{Rips complex} $\rips_{\eps}(X,d_X)$ is the abstract simplicial complex defined as
\[\rips_{\eps}(X,d_X)=\{A\subseteq X: \mbox{for all $x,x'\in A$, $d_X(x,x')\leq \eps$}\}. \]
Let $\simp$ be the category of abstract simplicial complexes and simplicial maps. The \emph{Rips filtration} is the functor $\rips_{\bullet}(X,d_X):\R\rightarrow \simp$ defined as 
\begin{align*}
    \eps \mapsto \rips_{\eps}(X,d_X), \ \mbox{and} \ \ 
\eps\leq \eps' \mapsto \rips_{\eps}(X,d_X)\hookrightarrow \rips_{\eps'}(X,d_X).
\end{align*}

\begin{definition}[Augmented metric spaces]\label{def:\filterMS} Let $(X,d_X)$ be a metric space and $f_X:X\rightarrow \R$ a function. We call the triple $\XX$ \emph{\afilterMS} (abbrev. \emph{\filterMSshort{}}). 

We say that  $\X$ is \emph{injective} if  $f_X:X\rightarrow \R$ is an injective function.
\end{definition}

Throughout this paper, every (augmented) metric space will be assumed to be finite. Let $\XX$ be \afilterMSshort. For $\sigma\in \R$, let $X_\sigma$  denote the sublevel set $f_X^{-1}(-\infty,\sigma]\subseteq X$. Let $(X_{\sigma},d_X)$ denote the restriction of the metric space $(X,d_X)$ to the subset $X_{\sigma}\subseteq X$. Similarly,  $(X_{\sigma},d_X,f_X)$ is the \filterMSshort{} obtained by restricting $d_X$ to $X_{\sigma}\times X_{\sigma}$ and $f_X$ to $X_{\sigma}$. The following 2-parameter filtration is considered in \cite{bauer2019cotorsion,carlsson2010multiparameter,carlsson2009theory} in the context of \emph{filtered} single linkage hierarchical clustering or \emph{filtered} persistent homology:

\begin{definition}[Rips bifiltration of \afilterMSshort]\label{def:Rips bifiltration} Let $\XX$ be \afilterMSshort. We define the \emph{Rips bifiltration} $\ripsbi_{\bullet}(\X):\R^2\rightarrow \simp$ of $\X$ as $(\eps,\sigma)\mapsto \rips_{\eps}(X_{\sigma},d_{X})$.
\end{definition}

By applying the $k$-th simplicial homology functor to the Rips bifiltration $\ripsbi_{\bullet}(\X)$, we obtain the persistence module $M:=\Ha_{k}(\ripsbi_{\bullet}(\X)):\R^2\rightarrow \vect$. 
Let $\Lcal$ denote the set of all lines of (strictly) positive slopes in $\R^2$. Given  $L\in \Lcal$, the restriction $M|_{L}:L\rightarrow \vect$ can be decomposed into the unique direct sum of interval modules over $L$ and thus we have the barcode $\barcp(M|_{L})$ of $M|_{L}$. The \emph{$k$-th fibered barcode} of $\X$ refers to the $\Lcal$-parametrized collection $\{\barcp(M|_{L})\}_{L\in \Lcal}$ {\cite{cerri2013betti,landi2018rank,lesnick2015interactive}}.

\paragraph{\Newbarcode{} for \afilterMSshort.} Let $(X,d_X)$ be a finite metric space. For $\eps\in[0,\infty)$, an \emph{$\eps$-chain} between $x,x'\in X$ stands for a sequence $x=x_1,x_2,\ldots,x_\ell=x'$ of points in $X$ such that $d_X(x_i,x_{i+1})\leq \eps$ for $i=1,\ldots,\ell-1$. 
Now given $\XX$ and $\sigma\in \RP$, consider a point $x\in X_\sigma$. Then for any $\eps \ge 0$, set $[x]_{\sigeps}$ as the collection of all points $x'\in X_{\sigma}$ that can be connected to $x$ through an $\eps$-chain  in $X_{\sigma}$.

The function $f_X: X \to \R$ induces an order on $X$:   consider any two $x, x'\in X$. If $f_X(x)<f_X(x')$, then we say that $x$ is \emph{older than} $x'$.

\begin{definition}[\Newbarcode{} for \afilterMSshort]\label{def:elder rule2}
Let $\XX$ be an \emph{injective} \filterMSshort. For each $x\in X$, we define \emph{its \newstair{}} as: 
\begin{align}\label{eq:barcode}
    I_{x}:&=\{(\sigma,\eps)\in \R^2: x\in X_{\sigma}\ \mbox{and $x$ is the oldest in $[x]_{(\sigma,\eps)}$ } \}
\end{align}
The collection $\I_{\X}:=\{I_{x}\}_{x\in X}$ is called the \emph{\newbarcode{} (\newbarcodeshort{} for short) of $\X$}.
\end{definition}

See Figure \ref{fig:elder-rule-staircode} for an example. The relationship between the \newbarcodeshort{} and the classic elder-rule will become clear in Section \ref{sec:treegrams}. 

\begin{figure}
    \centering
    \includegraphics[width=\textwidth]{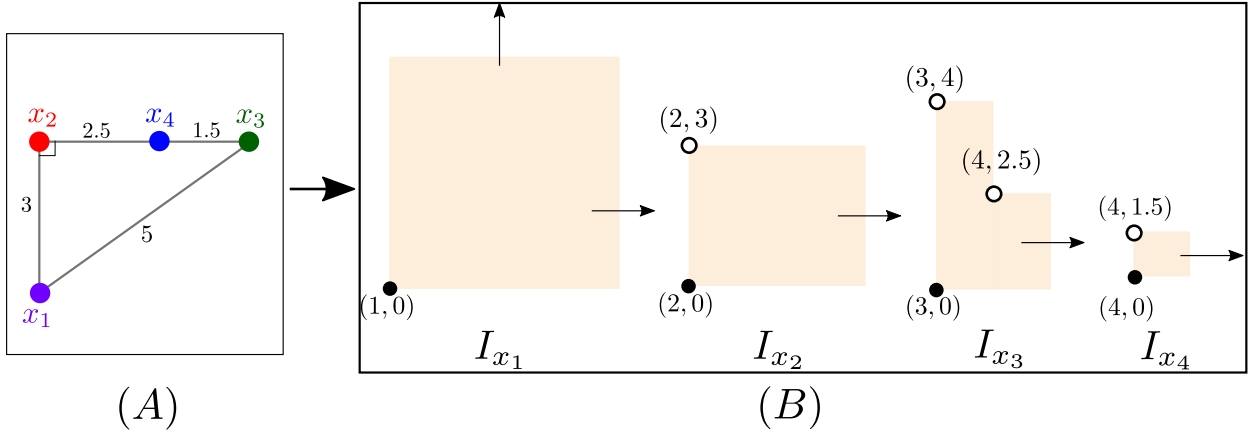}
    \caption{(A) Consider the triangle with edge lengths 3,4 and 5. Consider the \filterMSshort{} $\XX$ where $X:=\{x_1,x_2,x_3,x_4\}$, $d_X$ is the Euclidean metric on the plane, and $f_X$ is given as $f_X(x_i)=i$ for $i=1,2,3,4$.  (B) The \newbarcodeshort{} of $\X$.}
    \label{fig:elder-rule-staircode}
 
\end{figure}
\begin{definition}\label{def:staircase} An interval $I$ of $\R^2$ (Definition \ref{def:intervals}) is  a \emph{staircase interval (or simply staircase)} if there exists $(\sigma_0,\eps_0)\in I$ such that $(\sigma_0,\eps_0)\leq (\sigma,\eps)$ for all $(\sigma,\eps)\in I$, and $I$ is not bounded in the direction of $\sigma$-axis {(see Figure \ref{fig:corner types})}. 
\end{definition}
It turns out that each $I_{x}\in\I_{\X}$ is a staircase interval:

\begin{proposition}\label{prop:staircode is an interval} Each $I_x$ in Definition \ref{def:elder rule2} is a staircase interval (proof in Appendix \ref{sec:proofs}).
\end{proposition}

\paragraph{\Newstair{}s for non-injective case.} Even if $f_X$ is not injective, we still have the concept of the \newbarcodeshort{}. Consider \afilterMSshort{} $\XX$ such that $f_X$ is not injective. To induce the \newbarcodeshort{} of $\X$, we pick any order on $X$ which is \emph{compatible} with $f_X$: An order $<$ on $X$ is compatible with $f_X$ if $f_X(x)<f_X(x')$ implies $x<x'$ for all $x,x'\in X$. Now we define $\I_{\X}^<=\lmulti I_x^< :x\in X \rmulti$ where
 \begin{equation}\label{eq:order-induced block}
     I_x^{<}:=\{\sigeps\in \R^2: x\in X_{\sigma}\ \mbox{and $x=\min ([x]_{\sigeps},<)$}\}
 \end{equation}
 (we use double-curly-brackets $\lmulti- \rmulti$ to denote multisets). Regardless of the choice of $<$, the collection $\I_{\X}^<=\lmulti I_x^< :x\in X \rmulti$ satisfies all properties and theorems we prove later. Hence, for any possible compatible order $<$ we will refer to $\I_\X^<$ as an \emph{\newbarcodeshort{}} of $\X$.

\begin{example}[Constant function case]\label{ex:constant} Let $(X,d_X)$ be a metric space of $n$ points. Then, the barcode of $\hzero(\rips_{\bullet}(X,d_X)):\R\rightarrow \vect$ consists of $n$ intervals $J_i$, $i=1,\ldots, n$. Let $\XX$ be the \filterMSshort{} where $f_X$ is constant at $c\in \R$. Then, all possible total orders on $X$ are compatible with $f_X$ and all induce the same \newbarcodeshort{} $\I_{\X}=\lmulti [c,\infty)\times J_i:i=1,\ldots,n\rmulti$. 
\end{example}

In contrast to Example \ref{ex:constant}, different orders on $X$ in general induce different  \newbarcodeshort{}s of $\XX$ ; see Example \ref{ex:two barcodes}. Therefore, a single \newbarcodeshort{} of $\X$ is not necessarily an \emph{invariant} of $\X$, whereas the collection of all possible \newbarcodeshort{}s of $\X$ can be seen so (see item \ref{item:metrics and stability} in Section \ref{sec:conclusion}). This collection, however, is {\it not a complete invariant of $\X$} by the following reasoning: It is not difficult to find two non-isometric metric spaces $(X,d_X)$ and $(Y,d_Y)$ such that $\hzero(\rips_{\bullet}(X,d_X))$ and $\hzero(\rips_{\bullet}(Y,d_Y))$ have the same barcode. Let $f_X:X\rightarrow \R$ and $f_Y:Y\rightarrow \R$ be constant at $c\in \R$. Then, by Example \ref{ex:constant}, all the \newbarcodeshort{}s of $(X,d_X,f_X)$ and $(Y,d_Y,f_Y)$ (induced by all possible total orders on $X$ and $Y$) are the same (see item \ref{item:completeness} in Section \ref{sec:conclusion}).

We can recover the zeroth fibered barcode of \afilterMSshort{} $\X$ from its \newbarcodeshort{}: 
Computation of \newbarcodeshort{} and query time for fibered barcode are given in Theorem \ref{thm:computation}.

\begin{theorem}\label{thm:fibered barcodes}Let $\X$ be \afilterMSshort{} and let $M:=\hzero(\ripsbi_{\bullet}(\X))$. Let $\I_\X=\lmulti I_x:x\in X\rmulti$ be an \newbarcodeshort{} of $\X$. For each $L\in \Lcal$, the barcode $\barcp(M|_{L})$ coincides with the multiset
$\lmulti L\cap I_x:x\in X \rmulti$ (up to removal of empty sets, see Figure \ref{fig:sliced barcode}), (proof in Section \ref{proof:fibered barcodes}). 
\end{theorem}

\begin{figure}
    \centering
    \includegraphics[width=0.7\textwidth]{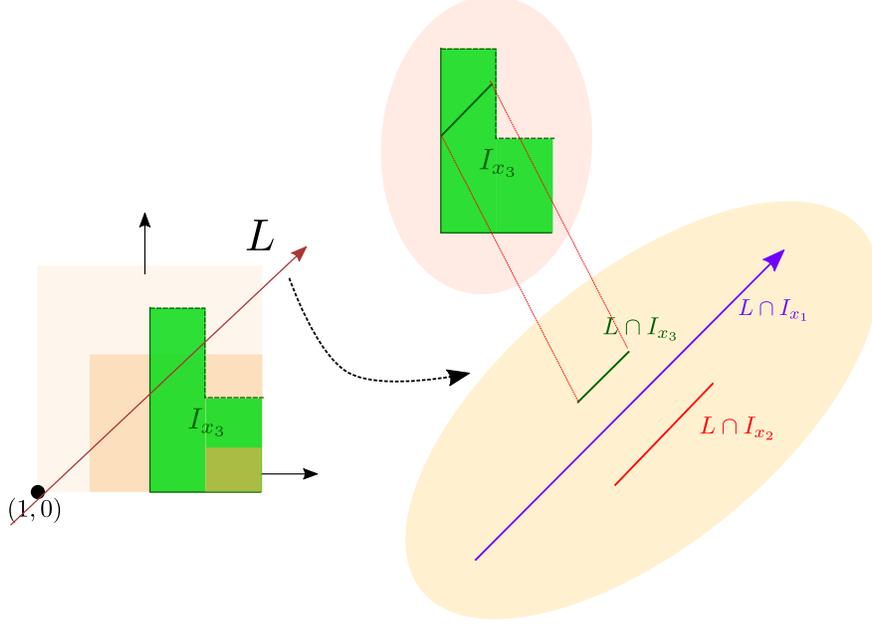}
    \caption{Left: The stack of $I_{x_i}$, $i=1,2,3,4$ from Figure \ref{fig:elder-rule-staircode} and a line $L\in \Lcal$ . Right: The barcode of $M|_{L}$. Since $L$ does not intersect $I_{x_4}$, only three intervals of $L\subset\R^2$ appear in the barcode.}
    \label{fig:sliced barcode}
\end{figure}

\begin{example}\label{ex:two barcodes}
If  \afilterMSshort{} is not injective, then there can be different \newbarcodeshort{}s w.r.t. different compatible orders. However, each of them will still be valid to produce the fibered barcodes. For example, let $(X,d_X)$ be the metric space in Figure \ref{fig:elder-rule-staircode} (A). Define $g_X:X\rightarrow \R$ by sending $x_1,x_2,x_3,x_4$ to $1,2,2,4$, respectively. Two orders $(x_1< x_2< x_3< x_4)$ and $(x_1<' x_3<' x_2<' x_4)$ are compatible with $g_X$. Consider the two \newbarcodesshort{} $\I_{\X}^<=\lmulti I_{x_i}^< :i=1,2,3,4\rmulti$ and $\I_{\X}^{<'}=\lmulti I_{x_i}^{<'}:i=1,2,3,4\rmulti$. While  $I_{x_i}^<=I_{x_i}^{<'}$ for $i=1,4$, the equality does not hold for $i=2,3$. However, both $\I_{\X}^<$ and $\I_{\X}^{<'}$ satisfy the statement in Theorem \ref{thm:fibered barcodes}. See Figure \ref{fig:comparison}.

\end{example}

\begin{figure}
    \centering
    \includegraphics[width=\textwidth]{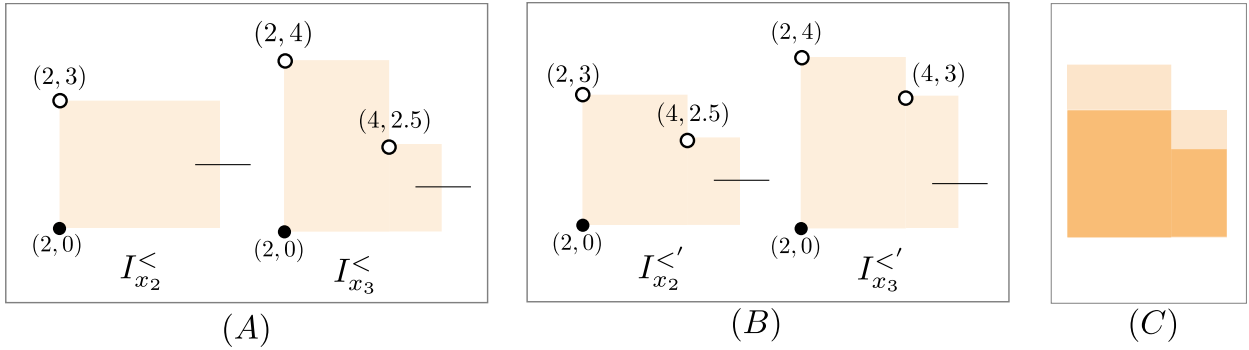}
    \caption{Illustration for Example \ref{ex:two barcodes}: (A) $I_{x_2}^<$ and $I_{x_3}^<$ . (B) $I_{x_2}^{<'}$ and $I_{x_3}^{<'}$. (C) Stack of $I_{x_2}^<$ and $I_{x_3}^<$. Stack of $I_{x_2}^{<'}$ and $I_{x_3}^{<'}$ look the same. Observe that for any $L\in \Lcal$, $\lmulti L\cap I_{x_2}^<,L\cap I_{x_3}^< \rmulti=\lmulti L\cap I_{x_2}^{<'},L\cap I_{x_3}^{<'}\rmulti$.}
    \label{fig:comparison}
\end{figure}

We will close this section with some definitions that will be useful later. 
Let $I$ be a staircase interval of $\R^2$. We define the three types of corner points as in Figure \ref{fig:corner types}. Roughly speaking, for each staircase $I_x$, type-0 is the left-bottom point; type-1 corners are those where the boundary transitions from a vertical segment to a horizontal one, while type-2 are those transitions from a horizontal one to vertical one (precise descriptions are given in Definition \ref{def:types of corners} of Appendix \ref{sec:proofs}).

Given a staircase interval $I$, for each $j=0,1,2$ we define the function $\gamma_j(I):\R^2\rightarrow \ZP$ as 
\begin{equation}\label{eq:gamma j}
    \gamma_j(I)(\ba)=\begin{cases} 1,&\mbox{$\ba$ is a $j$-th type corner point of $I$ } \\ 0,&\mbox{otherwise.}
\end{cases}
\end{equation}

\emph{Elder-rule feature functions} defined below will be useful in the later chapters.
\begin{definition}[Elder-rule feature functions]\label{def:elder rule feature functions}
Let $\X$ be \afilterMSshort{} and let $I_{\X}=\lmulti I_{x}:x\in X\rmulti$ be \anewbarcodeshort{} of $\X$. For $j=0,1,2$, we define the \emph{$j$-th elder-rule feature function} as the sum
$\gamma_j^{\X}=\sum_{x\in X} \gamma_j(I_{x}).$
\end{definition}

\begin{figure}\begin{center}
        \includegraphics[width=\textwidth]{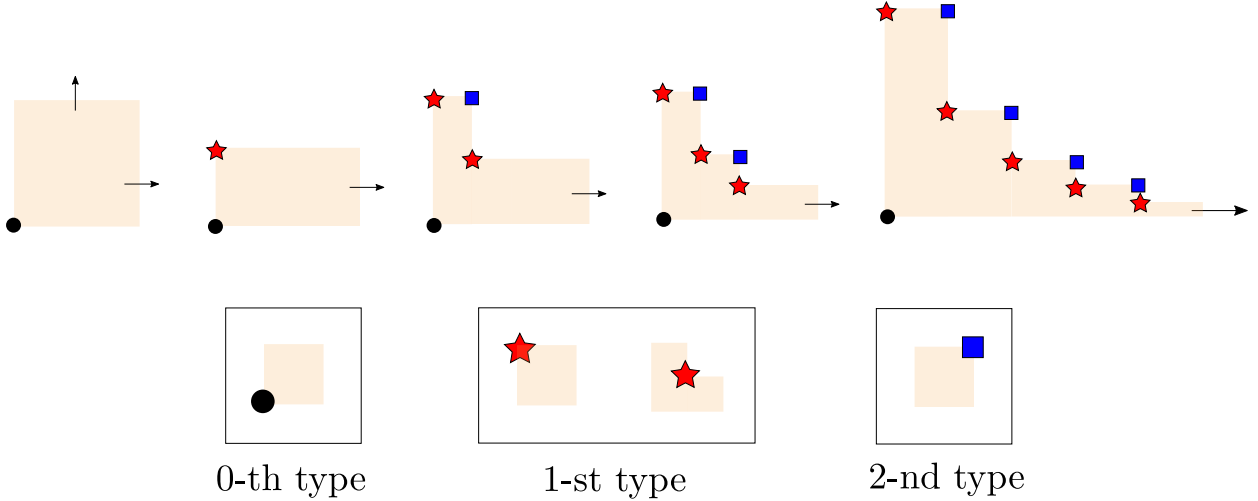}
    \caption{Every corner point of a staircase interval falls into three different types depending on its neighborhood information, as the pictures above illustrate. Staircase intervals in the first row are decorated by their corner points (a precise description is in Definition \ref{def:types of corners} of Appendix).}
    \end{center}
    \label{fig:corner types}
\end{figure}

\begin{remark}\label{rem:feature functions are graded Bettis}{It is not hard to check that $\gamma_j(I)$ in (\ref{eq:gamma j}) is equal to the $j$-th graded Betti number of the interval module $\R^2\rightarrow \vect$ supported by $I$ (Definition \ref{def:graded betti}). Thus $\gamma_j^{\X}=\sum_{x\in X}\beta_j^{I^{I_x}}$.}
\end{remark}

%%%%%%%%%%%%%%%%%%%%%%%%%%%%%%%%%%%%%%%%%%%%%%%%%%%%%%%%%%%%%%%%%%%%%%%%%%%%%%%%%%%%%%%%%%%%%%%%%%%%%%%%%%%%%%%%%%%%%%%%%%%%%%%%%%%%%%%%%%%%

\section{Decorated \newbarcodes{} and treegrams}\label{sec:decorated staircodes and treegrams}
In Section \ref{sec:treegrams} we prove Theorem \ref{thm:fibered barcodes} and introduce \emph{bipersistence treegrams} to encode multi-scale clustering information of  \filterMSshort{}s. In Section \ref{sec:decorated staircode} we show that an ``enriched'' \newbarcodeshort{} of \afilterMSshort{} $\X$ can recover the so-called \emph{fibered treegram}  of $\X$, i.e. 1-dimensional slices of the aforementioned bipersistence treegram. Also, we identify a sufficient condition on $\X$ for its \newbarcodeshort{} to be the barcode of the 2-parameter persistence module $\hzero(\ripsbi_{\bullet}(\X))$. {In Section \ref{sec:strong elder rule} we show that if $\hzero(\ripsbi_{\bullet}(\X))$ is interval decomposable, then its barcode is equal to the \newbarcodeshort{} of $\X$. Also, we stratify the collection of \filterMSshort{}s $\X$ according to the complexity of the  indecomposable summands of $\hzero(\ripsbi_{\bullet}(\X))$.}

\subsection{Bipersistence treegrams}\label{sec:treegrams}

\paragraph{Partitions and sub-partitions.} Let $X$ be a non-empty finite set.   We will call any partition $P$ of a subset $X'$ of $X$ a \emph{sub-partition} of $X$. In this case we call $X'$ the \emph{underlying set of $P$.}   
A partition of the empty set is defined as the empty set. By $\subpart(X)$, we denote the set of \emph{all sub-partitions of $X$}, i.e.	$\subpart(X):=\left\{P:  \exists X' \subseteq X\ \mbox{, $P$ is a partition of $X'$}\right\}.$ We refer to elements of a sub-partition of $X$ as \emph{blocks}.

Let $P,Q\in\subpart(X)$. By $P\leq Q$, we mean $P$ refines $Q$, i.e. for all $B\in P$, there exists $C\in Q$ such that $B\subseteq C$.
For example, let $X=\{x_1,x_2,x_3\}$ and consider the sub-partitions $P:=\{\{x_1\},\{x_2\}\}$ and $Q:=\{\{x_1,x_2\},\{x_3\}\}$ of $X$. Then, it is easy to see that $P\leq Q$. 

\emph{Treegrams} are a generalized notion of dendrograms \cite{smith2016hierarchical}, which are useful for visualizing the evolution of clustering information of 1-parameter simplicial filtrations:

\begin{definition}[Treegrams {\cite{smith2016hierarchical}}]\label{def:treegram} A \emph{treegram} over a finite set $X$ is any function $\theta_X:\R\rightarrow \subpart(X)
	$ such that the following properties hold:  (1) if $t_1\leq t_2$, then \hbox{$\theta_X(t_1) \leq \theta_X(t_2)$}, (2)  there exists $T>0$ such that $\theta_X(t)=\{X\}$ for $t\geq T$ and $\theta_X(t)$ is empty for $t\leq -T$, and (3) for all $t$ there exists $\epsilon>0$ s.t. $\theta_X(s)=\theta_X(t)$ for $s\in [t,t+\epsilon]$. See Figure \ref{fig:treegram}  for an example. Also, even when the domain $\R$ is replaced by any totally ordered set $L$ isomorphic to $\R$, $\theta_X$ is said to be a (1-parameter) treegram.
 \end{definition}

\begin{figure}
 	\begin{center} 		
 		\includegraphics[width=0.5\textwidth]{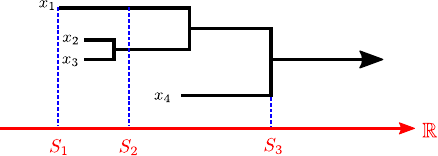}
 	\end{center}
 \caption{\label{fig:treegram} A (1D) treegram $\theta_X$ over the  set $X:=\{x_1,x_2,x_3,x_4\}$. Notice that $\theta_X(t)=\emptyset$ for $t\in(-\infty,S_1)$. Also,  $\theta_X(S_1)=\{\{x_1\}\}$, $\theta_X(S_2)=\{\{x_1\},\{x_2,x_3\}\}$, and $\theta_X(t)=\{X\}$ for all $t\in [S_3,\infty).$ }
 \end{figure}
 
Given a simplicial complex $K$, let $K^{(0)}$ be the vertex set of $K$. Let $\pi_0(K)$ be the partition of the vertex set $K^{(0)}$ according to the connected components of $K$. 
A functor $\K:\Pb\rightarrow \simp$ is said to be a \emph{filtration of $K$} if $\K(\ba)\subseteq K$ for all $\ba\in \Pb$,  every internal map is an inclusion, and there exists $\ba_0\in \Pb$ such that for all $\ba\in \Pb$ with $\ba_0\leq \ba$, $\K(\ba)=K$.

\begin{remark}[Treegrams induced by simplicial filtrations]\label{rem:treegrams arise from filtrations}  Let $K$ be a simplicial complex on the vertex set $X=\{x_1,x_2,\ldots,x_n\}$ and let $\K:\R\rightarrow \simp$ be a filtration of $K$. 
Assume that $K$ consists solely of one connected component, i.e. $\pi_0(K)=\{X\}$.  Then, the function $\pi_0(\K):\R\rightarrow \subpart(X)$ defined as $\eps\mapsto \pi_0(\K(\eps))$ is a treegram over $X$. 
\end{remark}

\paragraph{The zeroth elder rule for a 1-parameter filtration.} Let $\theta_X$ be a treegram over $X$. 
We define the \emph{birth time} of $x$ as $b(x):=\min\{\eps\in \R: \mbox{$x$ is in the underlying set of $\theta_X(\eps)$}\}
$
(by Definition \ref{def:treegram} (2) and (3), every $x\in X$ has the birth time $b(x)$).
Pick any order $<$ on $X$ such that $b(x)<b(x')$ implies $x<x'$ for all $x,x'\in X$.\footnote{This order $<$ is uniquely specified if all $x\in X$ have different birth times.} 
For $\eps\in [b(x),\infty)$, we denote the block to which $x$ belong in the sub-partition $\theta_X(\eps)$ by $[x]_\eps$. 
We define the \emph{death time} of $x$ as $d^<(x)=\sup\{\eps\in [b(x),\infty]: \mbox{$x=\min([x]_\eps,<)$}\}.$
As long as $<$ is compatible with the birth times, the \emph{elder-rule-barcode} is uniquely defined (which will be proved in Appendix): 

\begin{definition}[Elder-rule-barcode of a treegram]\label{def:elder rule barcode} Let $\theta_X:\R\rightarrow \subpart(X)$ be a treegram over $X$. For any order $<$ on $X$ compatible with the birth times, let $J_x:=\left[b(x),{d^{<}}(x)\right)$. The \elder{} of $\theta_X$ is defined as the multiset $\barc(\theta_X):=\lmulti J_x:x\in X \rmulti$.
\end{definition}
 
For the {\bf 1-parameter case}, the elder-rule-barcode of a treegram can be obtained by dismantling the treegram into linear pieces w.r.t. the elder rule -- see the theorem below. Even though this result is well-known (e.g, \cite{curry2018fiber}), we include a proof at the end of this section.

\begin{theorem}[Compatibility between the elder rule and algebraic decomposition]\label{thm:elder rule} 
Let $\K$ and $\theta_X$ be the filtration and the treegram in Remark \ref{rem:treegrams arise from filtrations}, respectively. Let $\barcp(\theta_X)=\lmulti J_{x}: x\in X \rmulti$ be the \elder{} of $\theta_X$.
Then, $\hzero(\K)\cong \bigoplus_{x\in X} \I^{J_{x}}$ (see Figure \ref{fig:elder rule introduction}).
\end{theorem}

\begin{figure}
    \centering
    \includegraphics[width=0.7\textwidth]{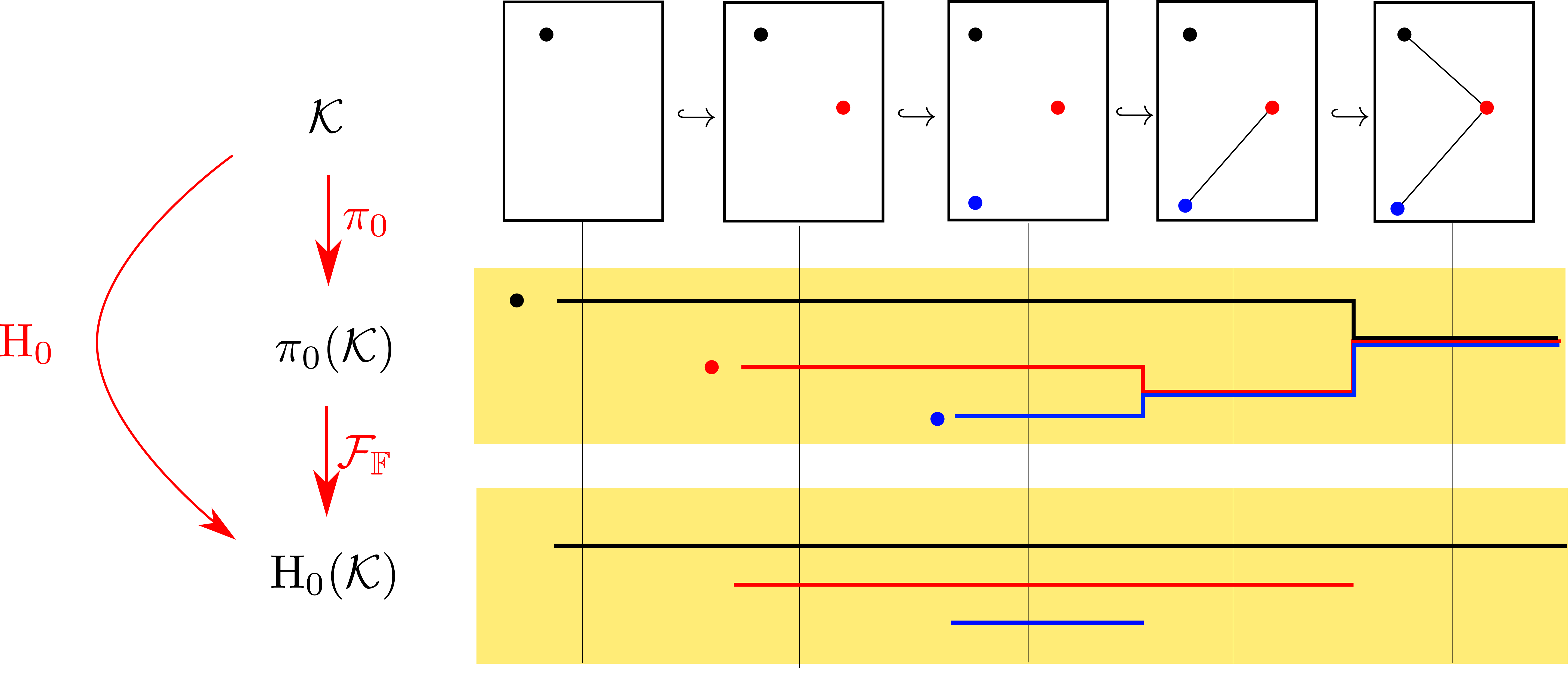}
    \caption{The first row represents a simplicial filtration $\K$. The second row stands for the the treegram $\pi_0(\K)$ which encodes the evolution of clusters in $\K$ (Remark \ref{rem:treegrams arise from filtrations}). The third row is the barcode of $\hzero(\K)$. The persistence module $\hzero(\K)$ can be obtained by applying the linearization functor (Definition \ref{def:free functor} in Appendix) to $\pi_0(\K)$. Alternatively, the barcode of $\hzero(\K)$ can also be obtained by applying the elder rule to $\pi_0(\K)$ (Definition \ref{def:elder rule barcode}).} 
    \label{fig:elder rule introduction}
\end{figure}

We are now ready to prove Theorem \ref{thm:fibered barcodes}.
\begin{proof}[Proof of Theorem \ref{thm:fibered barcodes}]\label{proof:fibered barcodes}  Fix $L\in \Lcal$. Since $L$ is isomorphic to $\R$ as a totally ordered set, $\K=\ripsbi_{\bullet}(\X)|_{L}:L\rightarrow \simp$ can be viewed as a 1-parameter filtration. Consider the treegram $\theta_X:=\pi_0(\K):L\rightarrow \subpart(X)$. By the definition of $I_x$s, it is clear that  $\lmulti L\cap I_x:x\in X \rmulti$ is the \elder{} of the treegram $\theta_X$ (Definition \ref{def:elder rule barcode}).  Hence, by Theorem \ref{thm:elder rule}, the multiset $\lmulti L\cap I_x:x\in X \rmulti$ is equal to the barcode of $\hzero\left(\K\right)$. Since $\hzero\left(\K\right)=M|_{L}$, we have $\lmulti L\cap I_x:x\in X \rmulti=\barcp(M|_{L})$.
\end{proof}

\paragraph{Bipersistence treegrams.} We now extend the notion of treegrams to encode the evolution of clusters of a 2-parameter filtration (similar ideas appear in \cite{kim2018spatio-temporal}). A \emph{bipersistence treegram} over a finite set $X$ is any function $\bitree:\R^2\rightarrow \subpart(X)$ such that  if $\ba\leq \bb$ in $\R^2$, then \hbox{$\bitree(\ba) \leq \bitree(\bb)$}.

We induce a {\bf bipersistence} treegram over $X$ from \afilterMSshort{} $\X$.

\begin{definition}[Rips bipersistence treegrams]\label{def:Rips bipersistence treegrams} Let $\XX$ be \afilterMSshort. We define $\bitreex:\R^2\rightarrow \subpart(X)$ as $\sigeps \mapsto \pi_0\left(\rips_{\eps}(X_{\sigma},d_{X})\right)$. This $\bitreex$ is said to be the \emph{Rips bipersistence treegram} of $\X$.
\end{definition}

Observe that $x\in X$ belongs to the underlying set of $\bitreex(\ba)$ if and only if $(f_X(x),0)\leq \ba$, i.e. $(f_X(x),0)$ is the \emph{birth grade} of $x$ in $\bitreex$. Assume that $\X$ is injective. Then the birth grades of elements in $X$ is totally ordered.  Note that the \newbarcodeshort{} of $\X$ can be extracted from $\bitreex$: Indeed, $I_x$ in equation (\ref{eq:barcode}) can be rephrased as 
$I_x=\{\sigeps\in \R^2:$ $x$ is in the underlying set of $\bitreex \sigeps$ and $x$ has the smallest birth grade in its block of $\bitreex \sigeps\}.$ See Figure \ref{fig:elder-rule-staircode1}.

\begin{figure}
    \centering
    \includegraphics[width=0.7\textwidth]{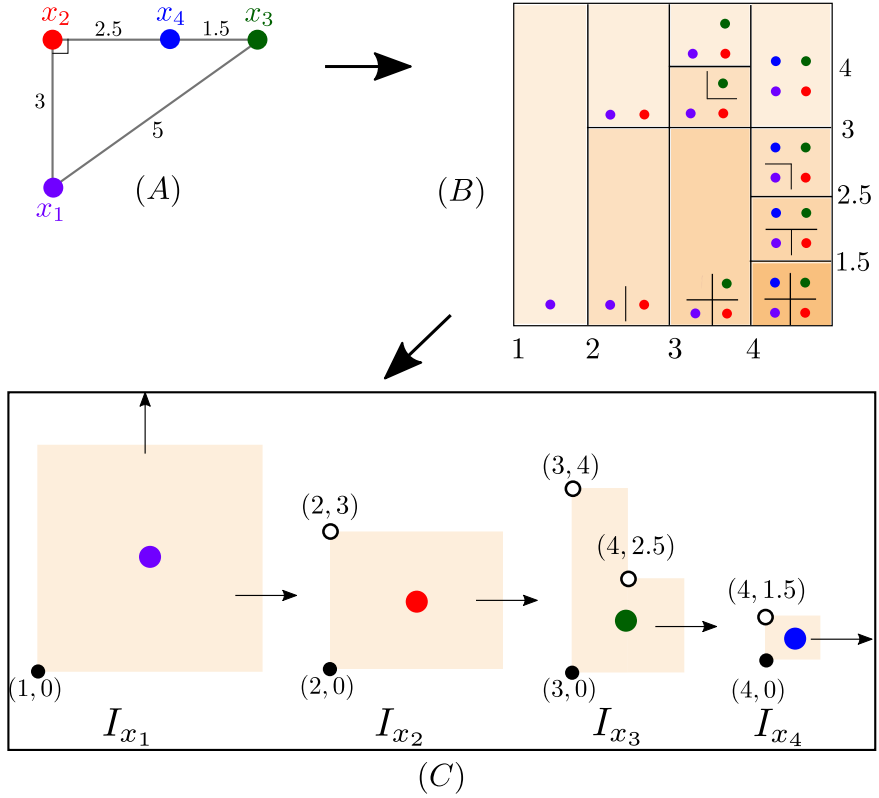}
    \caption{Consider the \filterMSshort{} $\X$ defined in Figure \ref{fig:elder-rule-staircode}. Figure (A) and (C) above are identical to Figure \ref{fig:elder-rule-staircode} (A) and (B), respectively. (B) The Rips bipersistence treegram of $\X$ (Definition \ref{def:Rips bipersistence treegrams}). The summarization processes (A)$\rightarrow$(B)$\rightarrow$(C) are analogous to the processes depicted in Figure \ref{fig:elder rule introduction}. Figures are best viewed in color.}    \label{fig:elder-rule-staircode1}
\end{figure}

\begin{definition}[Fibered treegrams]\label{def:fibered treegrams} Let $\bitreex$ be a Rips bipersistence treegram of \afilterMSshort{} $\X$. The \emph{fibered treegram} of $\bitreex$ refers to the collection $\{\bitreex|_{L}\}_{L\in \Lcal}$ of treegrams obtained by restricting $\bitreex$ to positive-slope lines (see Figure \ref{fig:fibered treegram} for an example).
\end{definition}

\begin{figure}
    \centering
    \includegraphics[width=0.8\textwidth]{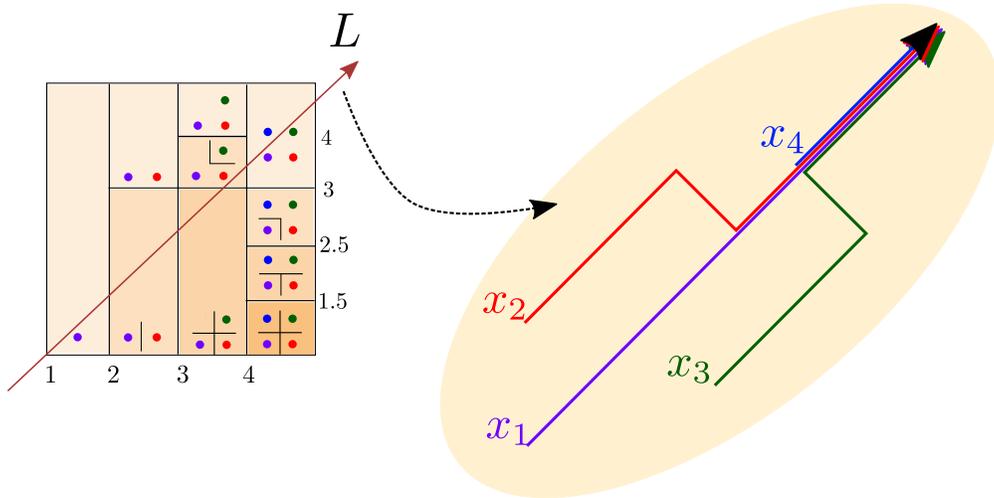}
    \caption{Consider the bipersistence treegram in Figure \ref{fig:elder-rule-staircode1} (B) and pick a line $L$ of positive slope. Then, we obtain a treegram over $L$.}
    \label{fig:fibered treegram}
\end{figure}

\paragraph{A combinatorial analogue of Theorem \ref{thm:betti formula}.}
Recall the elder-rule feature functions of \afilterMSshort{} $\X$ (Definition \ref{def:elder rule feature functions}). We will show that they can be used to retrieve the \emph{cardinality function} of $\bitreex$. 

\begin{definition}[Cardinality function] Let $\bitree$ be a bipersistence treegram over a set $X$. We call the function $\abs{\bitree}:\R^2\rightarrow \ZP$ defined as $\ba\mapsto \abs{\bitree(\ba)}$, the \emph{cardinality function} of $\bitree$.
\end{definition}

For $A\subseteq \R^2$ we define the indicator function $\one_A:\R^2\rightarrow \ZP$ of $A$ as \[\one_{A}(\ba):=\begin{cases}1,&\ba\in A, \\ 0,&\mbox{otherwise.}
\end{cases}\]

 The following proposition directly follows \cite[Proposition 32]{dey2018computing}:

\begin{proposition}\label{prop:cheng} Let $I$ be a staircase interval. Then,  
$\one_{I}(\ba)=\sum_{\bx\leq \ba}\sum_{j=0}^2(-1)^{j}\gamma_j(I)(\bx).$
\end{proposition}

The \newbarcodeshort{} and elder-rule feature functions of \afilterMSshort{} $\X$ recovers the cardinality function of $\bitreex$, which is analogous to Theorem \ref{thm:betti formula}:

\begin{theorem}\label{thm:elder rule barcode recovers the cardinality function} Let $\X$ be \afilterMSshort{} and let $I_{\X}=\lmulti I_{x}:x\in X\rmulti$ be \anewbarcodeshort{} of $\X$. For each $\ba \in \R^2$, 
\begin{align}
\abs{\bitreex(\ba)}&=
\mbox{(The number of intervals $I_{x}\in \I_{\X}$ containing $\ba$).}   \label{eq:first} \\&=\sum_{\bx\leq\ba}\sum_{j=0}^{2}(-1)^j\gamma_j^\X(\bx).\label{eq:second}
\end{align}
\end{theorem}

 \begin{proof}
 \label{proof:elder rule barcode recovers the cardinality function}For simplicity we assume the injectivity of $\X$. We prove the equality in (\ref{eq:first}). Let $\sigeps\in \R^2$. Since each block in $\bitreex\sigeps$ contains its unique oldest element,  $\abs{\bitreex\sigeps}$ is equal to the cardinality of the set 
\[A\sigeps:=\{x\in X_{\sigma}:x\ \mbox{is the oldest in the block containing $x$ in $\theta_X\sigeps$}\}.\]
By equation (\ref{eq:barcode}), $\ba$ belongs to $I_{x_i}$ if and only if $x_i \in A\sigeps$, implying the equality \[\abs{A\sigeps}=\mbox{(The number of intervals $I_{x_i}\in \I_{\X}$ containing $\sigeps$),}\] 
as desired. The equality in (\ref{eq:second}) directly follows from Proposition \ref{prop:cheng} and Definition \ref{def:elder rule feature functions}.
 \end{proof}

%%%%%%%%%%%%%%%%%%%%%%%%%%%%%%%%%%%%%%%%%%%%%%%%%%%%%%%%%%%%%%%%%%%%%%%%%%%%%%%%%%%%%%%%%%%%%%%%%%%%%%%%%%%%%%%%%%%%%%%%%%%%%%%%%%%%%%%%%%%%%%%%%%%%%%%%%%%%%%%%%%%%%%%%%%%%%
\subsection{\Newbarcodes{} and fibered treegrams}\label{sec:decorated staircode}

In this section we identify a sufficient condition on \afilterMSshort{} $\X$ for its \newbarcodeshort{} to coincide with the barcode of the 2-parameter persistence module $\hzero(\ripsbi_{\bullet}(\X))$ (Theorem \ref{thm:elder rule2}). Also, in general, all fibered treegrams can be recovered from \newbarcodeshort{}s (Theorem \ref{thm:fibered treegrams}). 

Let $(X,d_X)$ be a  metric space and fix $x,x'\in X$. Recall that an \emph{$\eps$-chain} between $x$ and $x'$ is a finite sequence
$x=x_1,x_2,\ldots,x_\ell=x'$ in $X$ where each consecutive pair $x_i,x_{i+1}$ is within distance $\eps$. Define (in fact an ultrametric) $u_X:X\times X\rightarrow \RP$ as
\begin{equation}\label{eq:ultrametric}
    u_X(x,x'):=\min\{\eps\in[0,\infty): \mbox{there exists an $\eps$-chain between $x$ and $x'$}\} \ \mbox{(see \cite{carlsson2010characterization})}.
\end{equation}

For a metric space $(X,d_X)$ and pick any total order $<$ on $X$. Let $x\in X$ be a non-minimal element of $(X,<)$. A $<$-\emph{conqueror} of $x$ is an element $x'\in X$ such that (1) $x'<x$, and (2) for any $x''\in X$ with $x''<x$, it holds that $u_X(x,x')\leq u_X(x,x'')$.

Now consider an \filterMSshort{} $\XX$. A $<$-\emph{conqueror} function $c_x:[f_X(x),\infty)\rightarrow X$ of a non-minimal $x\in X$ sends $\sigma\in[f_X(x),\infty)$ to a conqueror of $x$ in $(X_{\sigma},d_X)$. For the minimum $x'\in (X,<)$, define $c_{x'}:[f_X(x'),\infty)\rightarrow X$ to be the constant function at $x'$. 

We generalize Theorem \ref{thm:elder rule} and at the same time strengthen Theorem \ref{thm:fibered barcodes} for 2-parameter persistence modules induced by a special type of \filterMSsshort:

\begin{theorem}[Compatibility between the \newbarcodeshort{}s and algebraic decomposition]\label{thm:elder rule2}Let $\XX$ be \afilterMSshort{} and fix any order $<$ on $X$ compatible with $f_X$. 
Assume that there exists a \emph{constant} $<$-conqueror function for \emph{each} $x\in X$.\footnote{Observe that if this property holds for the order $<$, then the same property holds for any other order  $<'$ that is compatible with $f_X$, and $\I_\X^{<}=\I_{\X}^{<'}$. } Then, $\hzero\left(\ripsbi_{\bullet}(\X)\right)$ is interval decomposable and its barcode coincides with the \newbarcodeshort{} $\I_{\X}^<$.
\end{theorem}

The proof of Theorem \ref{thm:elder rule2} is similar to that of Theorem \ref{thm:elder rule}. The both proofs are given at the end of this section. Consider the \filterMSshort{} $\X$ in Figure \ref{fig:elder-rule-staircode}. Observe that $\X$ satisfies the assumption in Theorem \ref{thm:elder rule2}. Therefore, $\hzero\left(\ripsbi_{\bullet}(\X) \right)$ is interval decomposable. 
There exists a class of \filterMSshort{}s to which Theorem \ref{thm:elder rule2} applies as shown by the following corollary.
\begin{corollary}\label{cor:ultrametric}Let $\XX$ be any \filterMSshort{} where $d_X$ is an ultrametric, i.e. $d_X(x,x'')\leq\max\left(d_X(x,x'),d_X(x',x'')\right)$ for all $x,x',x''\in X$. Then, $\hzero\left(\ripsbi_{\bullet}(\X)\right)$ is interval decomposable (in fact, its barcode consists solely of infinite rectangular intervals).
\end{corollary}
\begin{proof}Let $<$ be an order on $X$ which is compatible with $f_X$. For each non-minimal $x\in (X,<)$, pick an $x'\in X$  such that (1) $x'<x$, and (2) for any $x''\in X$ with $x''<x$, it holds that $d_X(x,x')\leq d_X(x,x'')$. Now observe that $x'$ is a $<$-conqueror in $(X_{\sigma},d_X)$ for \emph{every} $\sigma\in [f_X(x),\infty)$, completing the proof.
\end{proof}

The converse of Theorem \ref{thm:elder rule2} is false by virtue of the following example.
\begin{example}\label{ex:counter-example}Let $X:=\{x_i\}_{i=1}^8$. Consider $\XX$ where $(X,d_X)$ is depicted in Figure \ref{fig:counter-example} and $f_X(x_i)=i$ for each $i=1,\ldots,8$. Then, $\hzero(\ripsbi(\X))$ is interval decomposable even though $x_6\in X$ does not have a constant conqueror. See below for the proofs of these claims.
\end{example}

\begin{proof}[Details from Example \ref{ex:counter-example}]
The fact that $x_6$ does not have a constant conqueror can be ascertained from the following observation: For $\sigma\in [6,7)$, $x_1,x_2$ and $x_3$ are the conquerors of $x_6$ in $X_{\sigma}$. For $\sigma\in [7,8)$, $x_3,x_4$ and $x_5$ are the conquerors of $x_6$ in $X_{\sigma}$. For $\sigma\in [8,\infty)$, $x_5$ is the unique conqueror of $x_6$ in $X_{\sigma}$.

Let $\I_{\X}=\{I_{x_i}\}_{i=1}^8$ be the ER-staircode of $\X$. To prove that $M:=\hzero(\ripsbi(\X))$ is interval decomposable, it suffices to construct an isomorphism $f$ from 
$N:=\bigoplus_{i=1}^8 I^{I_{x_i}}$ to $M$. For $i=1,\ldots,8$ and for $(\sigma,\eps)\in [i,\infty)\times \R_+$, let $[x_i]_{(\sigma,\eps)}$ be the zeroth homology class of $x_i$. When confusion is unlikely, we will suppress the subscript $(\sigma,\eps)$ in $[x_i]_{(\sigma,\eps)}$.

For each $i$, consider $1_i:=1\in (I^{I_{x_i}})_{(i,0)}(=\F)$. We declare that
\begin{align*}
    1_1 &\stackrel{f_{(1,0)}}{\longmapsto} [x_1]&    1_2 &\stackrel{f_{(2,0)}}{\longmapsto} [x_2]-[x_1]\\
    1_3 &\stackrel{f_{(3,0)}}{\longmapsto} [x_3]-[x_1]&
    1_4 &\stackrel{f_{(4,0)}}{\longmapsto} [x_4]-[x_3]\\
    1_5 &\stackrel{f_{(5,0)}}{\longmapsto} [x_5]-[x_4]&
    1_6 &\stackrel{f_{(6,0)}}{\longmapsto} [x_2]-[x_1]+ [x_4]-[x_3]+ [x_6]-[x_5]\\
    1_7 &\stackrel{f_{(7,0)}}{\longmapsto} [x_7]-[x_3]&
    1_8 &\stackrel{f_{(8,0)}}{\longmapsto} [x_8]-[x_6].
\end{align*}
Since $\{1_i:i=1,\ldots,8\}$ is a set of all generators of $N$, the above specification gives rise to a unique morphism $f:N\rightarrow M$.
It is not hard to check that $f$ is actually an \emph{isomorphism}.
\end{proof}

\begin{figure}
    \centering
    \includegraphics[width=0.7\textwidth]{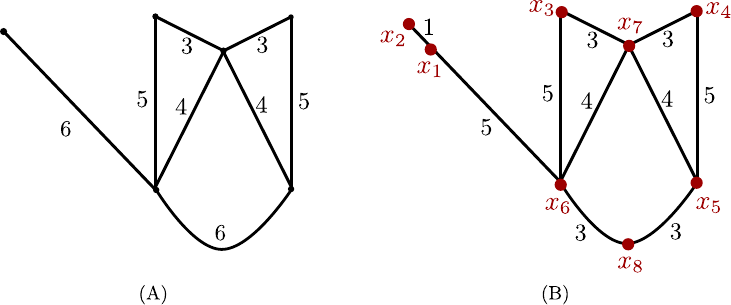}
    \caption{(A) A metric graph $G$. The distance between any two points  on $G$ is the length of a shortest path connecting them. (B) The embedding of $(X,d_X)$ in $G$.}
    \label{fig:counter-example}
\end{figure}

We enrich the \newbarcodeshort{} in order to query the fibered treegram: Let $\XX$ be \afilterMSshort. Let $<$ be any order on $X$ which is compatible with $f_X$. For each $x$, we define $I_x^\ast$ as the pair $(I_x,c_x)$ of the set $I_x$  and the $<$-conqueror function $c_x$.   
The collection $\I_{\X}^\ast:=\{I_{x}^\ast\}_{x\in X}$ is said to be the \emph{decorated \newbarcodeshort{}} of $\X$. See Figure \ref{fig:decorated staircodes}. The following result is easy to obtain with the help of decorations. 

\begin{theorem}\label{thm:fibered treegrams}Given any $L\in \Lcal$, the fibered treegram $\bitreex|_L$ can be recovered from the decorated \newbarcodeshort{} $\I_{\X}^\ast$ of the \filterMSshort{} $\XX$.
\end{theorem}

\begin{figure}
    \centering
    \includegraphics[width=0.7\textwidth]{decorated_staircodes.pdf}
    \caption{Decorated intervals corresponding to the four intervals in Figure \ref{fig:elder-rule-staircode} (C). For each $i=2,3,4$, the upper boundary of $I_{x_i}$ is decorated by the conqueror of $x_i$.}
    \label{fig:decorated staircodes}
\end{figure}

\paragraph{Proofs of Theorems \ref{thm:elder rule} and \ref{thm:elder rule2}.} We first define the \emph{linearization functor}:

\begin{definition}[Linearization functor]\label{def:free functor} Let $X$ be a non-empty finite set. We define the linearization functor $\free:\subpart(X)\rightarrow \vect$ as follows. 
\begin{enumerate}[label=(\roman*)]
    \item Each $P\in \subpart(X)$ is sent to the vector space $\free(B)$ which consists of formal linear combinations of elements of $P$  over the field $\F$. In other words, \[\free(P)=\left\{\displaystyle \sum_{B\in P}c_{B}B: \ c_B\in \F\right\}.\] 
    By identifying each $B\in P$ with $1\cdot B\in \free(P)$, the sub-partition $P$ can be viewed as a basis of $\free(P)$.
    \item Each pair $P\leq Q$ in $\subpart(X)$ is sent to the linear map $\free(P)\rightarrow \free(Q)$ which sends each $1\cdot B\in \free(P)$ to $1\cdot B'\in \free(Q)$ such that $B\subseteq B'$.
\end{enumerate}
\end{definition}

The following proposition is straightforward by \cite[Theorem 7.1]{munkres1984elements}:

\begin{proposition}\label{prop:trivially isomorphic} 
\begin{enumerate}[leftmargin=*,label=(\roman*)]
    \item Let $\theta_X:\R\rightarrow\subpart(X)$ be the treegram obtained by applying $\pi_0$ to a filtration $\K:\R\rightarrow \simp$ (Remark \ref{rem:treegrams arise from filtrations}). The two 1-parameter persistence modules $\free\circ \theta_X$ and $\hzero(\K)$ are isomorphic.\label{item:trivially isomorphic1}
    \item Let $\X$ be \afilterMSshort.  The two 2-parameter persistence modules $\free\circ \bitreex$ and $\hzero\left(\ripsbi_{\bullet}(\X)\right)$ (Definitions \ref{def:Rips bifiltration} and \ref{def:Rips bipersistence treegrams}) are isomorphic.\label{item:trivially isomorphic2}
\end{enumerate}
\end{proposition}

Now we are ready to prove Theorems \ref{thm:elder rule} and \ref{thm:elder rule2}.

\begin{proof}[Proof of Theorem \ref{thm:elder rule}]\label{proof:elder rule}  Without loss of generality, let $X=\{x_1,\ldots,x_n\}$. By Proposition \ref{prop:trivially isomorphic} \ref{item:trivially isomorphic1}, $\hzero(\K)$ is isomorphic to  $M:=\free\circ\theta_X$, and thus it suffices to show that $M\cong \bigoplus_{i=1}^n I^{[b(x_i),d(x_i))}=:N$. We may assume that $b(x_1)\leq b(x_2)\leq \ldots \leq b(x_n)$. For each $i\in \{2,3,\ldots,n\}$, 
we pick a certain $x_{q(i)}$ which merges with $x_i$ earliest in the treegram $\theta_X$ among all the points in $\{x_1,x_2,\ldots,x_{i-1}\}$. This defines a function $q:\{2,3,\ldots,n\}\rightarrow \{1,2,\ldots,n\}$ (such function $q$ is not necessarily unique, since some two points $x_{j_1},x_{j_2}$ might merge with another point $x_{j_3}$ at the same time).

For $x_i\in X$ and $\eps\in[b(x_i),\infty)$, let $[x_i]_\eps$ be the block containing $x_i$ in the sub-partition $\theta_X(\eps)$ of $X$.

On the interval $(-\infty,b(x_1))$, both $M$ and $N$ are trivial and thus let $f_\eps$ be the zero map for $\eps\in (-\infty,b(x_1))$.

Fix  $\eps\in [b(x_1),\infty)$. Note that the vector space $M(\eps)$ is spanned by $\A=\{[x_i]_{\eps}\in\theta_X(\eps):b(x_i)\leq \eps\}$. Therefore, $M(\eps)$ is also spanned by $\B=\{[x_i]_{\eps}-[x_{q(i)}]_{\eps}:b(x_i)\leq \eps\}$, which is obtained by applying elementary linear operations on $\A$. Furthermore, observe that 
\[\B'=\{[x_1]_{\eps}\}\cup \left(\{[x_i]_{\eps}-[x_{q(i)}]_{\eps}:b(x_i)\leq \eps\}\setminus\{0\}\right)\] 
is a linearly independent set and in turn a basis of $M({\eps})$. Define the linear map $f_\eps:M(\eps)\rightarrow N(\eps)$ by defining it on the basis $\B'$ as follows:
\begin{enumerate}[leftmargin=*,label=(\roman*)]
    \item send $[x_1]_{\eps}$ to $1$ in the 1-st summand of $N(\eps)=\bigoplus_{i=1}^n I^{[b(x_i),d(x_i))}(\eps)$.
    \item send each basis element $[x_i]_{\eps}-[x_{q(i)}]_{\eps}(\neq 0)$ to $1$ in the $i$-th summand of 
    \[\bigoplus_{i=1}^n I^{[b(x_i),d(x_i))}(\eps).\]
\end{enumerate}
Then, one can check that the collection $f=\{f_\eps\}_{\eps\in \R}$ is an isomorphism between $M$ and  $N$, as desired.
\end{proof}

We make use of the same strategy as Theorem \ref{thm:elder rule} for proving Theorem \ref{thm:elder rule2}:
 
\begin{proof}[Proof of Theorem \ref{thm:elder rule2}]\label{proof:elder rule2}
Without loss of generality, we may assume that $X=\{x_1,\ldots,x_n\}$,  $f_X(x_1)\leq f_X(x_2)\leq \ldots \leq f_X(x_n)$, and let the order $<$ on $X$ defined as  $(x_1<x_2<\ldots<x_n)$. Also, assume that each $<$-conqueror function $c_{x_i}:\R\rightarrow X$ is constant at $q(i)\in X$ (then by definition $q(1)=x_1$). By Proposition \ref{prop:trivially isomorphic} \ref{item:trivially isomorphic2}, it suffices to show that $M:=\free\circ\bitreex$ is isomorphic to $N=\bigoplus_{i=1}^n I^{I_{x_i}^<}$.

For $x_i\in X$, and $\sigeps\in \R^2$ with $\sigeps\geq (f(x_i),0)$, let $[x_i]_{\sigeps}$ be the block containing $x_i$ in the sub-partition $\bitreex \sigeps$ of $X$.

For any $\sigeps\in \R^2$ such that $\sigeps \not\geq(f_X(x_1),0)$, both  $M\sigeps$ and $N\sigeps$ are trivial and thus let $f_{\sigeps}$ be the zero map for $\sigeps \not\geq(f_X(x_1),0)$.

Fix  $\sigeps\in \R^2$ such that $\sigeps \geq(f_X(x_1),0)$.  The vector space $M\sigeps$ is spanned by $\A=\{[x_i]_{\sigeps}\in \bitreex\sigeps: (f_X(x_i),0)\leq\sigeps \}$. Therefore, $M\sigeps$ is also spanned by $\B=\{[x_1]_{\sigeps}\}\cup \{[x_i]_{\sigeps}-[x_{q(i)}]_{\sigeps}:(f_X(x_i),0)\leq\sigeps\}$, which is obtained by applying elementary linear operations on $\A$. Furthermore, note that \[\B':=\{[x_1]_{\sigeps}\}\cup\left(\{[x_i]_{\sigeps}-[x_{q(i)}]_{\sigeps}:(f_X(x_i),0)\leq\sigeps\}\setminus\{0\}\right)\] is a linearly independent set and in turn a basis of $M\sigeps$. Let us define a linear map $f_{\sigeps}:M\sigeps\rightarrow N\sigeps$ by defining it on the basis $\B'$ as follows:
\begin{enumerate}[label=(\roman*)]
    \item send $[x_1]_{\sigeps}$ to $1$ in the 1-st summand of $N\sigeps=\bigoplus_{i=1}^n I^{I_{x_i}}\sigeps$.
    \item send each basis element $[x_i]_{\sigeps}-[x_{q(i)}]_{\sigeps}(\neq 0)$ to $1$ in the $i$-th summand of $N\sigeps=\bigoplus_{i=1}^n I^{I_{x_i}}\sigeps$.
\end{enumerate}

By invoking the construction of the $<$-conqueror functions $c_{x_i}$ and the \newbarcodeshort{} $\I_{\X}^<=\lmulti I_{x_i}^<:i=1,\ldots,n\rmulti$, one can check that the collection $f=\{f_{\sigeps}\}_{\sigeps\in \R^2}$ is an \emph{isomorphism} between $M$ and $N$, as desired.
\end{proof}

\subsection{{\Newbarcode{}s} are the only candidates for the barcodes}\label{sec:strong elder rule} 

The compatibility between the elder-rule and the algebraic decomposition theory (Theorem \ref{thm:elder rule2}) will be enhanced to Theorem \ref{thm:strong elder rule} below.  For any $(\sigma_0,\eps_0)\in \R^2$, let $U(\sigma_0,\eps_0):=\{(\sigma,\eps)\in \R^2:(\sigma_0,\eps_0)\leq (\sigma,\eps)\}$, i.e. the closed quadrant whose lower-left corner point is $(\sigma_0,\eps_0)$. 
\begin{theorem}\label{thm:strong elder rule} Let $\X$ be an injective \filterMSshort{} such that $M:=\hzero\left(\ripsbi_{\bullet}(\X)\right)$ is interval decomposable. 
Then, the barcode of $M$ coincides with the \newbarcodeshort{} $\I_\X^<$ of $\X$.  
\end{theorem}

The proof utilizes results in Section \ref{sec:staircodes and Bettis} and thus is deferred to that section.

\begin{remark}\label{rem:testing interval decomposability} By Theorem \ref{thm:strong elder rule}, testing the interval decomposability of $\hzero(\ripsbi(\X))$ is equivalent to testing whether $\hzero(\ripsbi(\X))\cong N:=\bigoplus_{i=1}^nI^{I_{x_i}}$. In \cite{brooksbank2008testing} there exists a deterministic algorithm
for testing such an isomorphism.
\end{remark}

\paragraph{Stratification of the collection of \filterMS{}s} Let us consider the following collections of \filterMSshort{}s.
\begin{enumerate}[leftmargin=*, label=(\roman*)]
    \item $\aug$ is defined as the collection of all finite \filterMSshort{}s.
\end{enumerate}
The following are sub-collections of $\aug$.
\begin{enumerate}[resume,leftmargin=*, label=(\roman*)]
    \item $\ult$ consists of all finite \filterMSshort{}s $(X,d_X,f_X)$ where $d_X$ is an ultrametric.
    
    \item $\rep$ consists of all finite \filterMSshort{}s $\X$ such that the horizontal internal maps of  $\hzero\left(\ripsbi_{\bullet}(\X)\right)$ are injective.
    
    \item $\rec$ consists of all finite \filterMSshort{}s $\X$ such that $\hzero\left(\ripsbi_{\bullet}(\X)\right)$ is rectangle decomposable, i.e. each indecomposable summand is $I^{[a,b)\times [c,d)}$ for some intervals $[a,b),[c,d)$ of $\R$.
    
    \item $\eld$ consists of all finite \filterMSshort{}s $\X$ such that the assumption of Theorem \ref{thm:elder rule2} holds (and thus interval decomposable).
    
    \item $\stair$ consists of all finite \filterMSshort{}s $\X$ such that $\hzero\left(\ripsbi_{\bullet}(\X)\right)$ is interval decomposable.
\end{enumerate}
In order to clarify the relationship among these collections, we begin by recalling: 
\begin{theorem}[{\cite[Corollary 3.17]{bauer2019cotorsion}}]\label{thm:bauer} $\rep= \rec$.
\end{theorem}
We enrich Theorem \ref{thm:bauer} as follows:

\begin{theorem} \label{thm:stratification} $\ult \subsetneq \rep = \rec \subsetneq \eld \subsetneq \stair \subsetneq \aug.$ 
\end{theorem}

{We in particular remark that Example \ref{ex:non-decomposable}  provides an  \filterMSshort{} which does not belong to  $\stair$. Such  examples provide clues for how to construct \filterMSshort{}s $\X$ which yield $\hzero(\ripsbi(\X))$ whose isomorphism type is exotic, thus complementing the results of   \cite{bauer2019cotorsion}.}

\begin{proof}
\begin{enumerate}[leftmargin=*, label=(\roman*)]
    \item $\ult \subseteq \rep$:  Consider \afilterMSshort{} $\XX$ where $d_X$ is an ultrametric. By Proposition \ref{prop:trivially isomorphic} \ref{item:trivially isomorphic2}, it suffices to show that every horizontal internal map of $\bitree:\R^2\rightarrow \subpart(X)$ is injective. Pick $(\sigma_1,\eps), (\sigma_2,\eps)\in \R^2$ with $\sigma_1\leq \sigma_2$ and pick $x,y\in X$ with $f_X(x),f_X(y)\leq \sigma_1$. Assume that $[x]_{(\sigma_2,\eps)}=[y]_{(\sigma_2,\eps)}$ and let us show that $[x]_{(\sigma_1,\eps)}=[y]_{(\sigma_1,\eps)}$. The assumption implies that there exists a sequence $x=x_0,\ldots,x_n=y$ in $X_{\sigma_2}$ such that $d_{X}(x_i,x_{i+1})\leq \eps$ for each $i$. Since $d_X$ is an ultrametric, we have that $d_X(x,y)\leq \max_{i=0}^{n-1} d_X(x_i,x_{i+1})\leq \eps$. Invoking $f_X(x),f_X(y)\leq\sigma_1$, we have $[x]_{(\sigma_1,\eps)}=[y]_{(\sigma_1,\eps)}$, as desired.
    
    \item $\ult \neq \rep$: Let us equip the set $X:=\{1,2,3\}$ with the standard metric $d(i,j):=\abs{i-j}$, $i,j\in\{1,2,3\}$, and the map $f:X\rightarrow \R$ defined as $i\mapsto i$ for $i=1,2,3$. Observe that $d_X$ is not an ultrametric, but every horizontal internal map of $\hzero(\ripsbi(\X))$ is injective.
    
    \item $\rep \subseteq \eld$: Consider \afilterMSshort{} $\XX$ in $\rep$. Pick an order $<$ on $X$ which is compatible with $f_X$. Let $x\in (X,<)$ be a non-minimal element and let $\sigma_0:=f_X(x)$. Let $x'$ be a conqueror of $x$ in the metric space $(X_{\sigma_0},d_X)$. It suffices to show that for each $\sigma\in [\sigma_0,\infty)$, $x'$ is a conqueror of $x$ in $(X_{\sigma},d_X)$. Fix $\sigma\in [\sigma_0,\infty)$. Let $x''\in X_{\sigma}$ be a conqueror of $x$ in $(X_\sigma,d_X)$. Let $u_X^{\sigma}:X\times X\rightarrow \R$ be the (ultra)metric induced by $(X_\sigma,d_X)$ as in (\ref{eq:ultrametric}). Let $\eps:=u_X^{\sigma}(x'',x)$. By definition of $x''$, we have 
\begin{equation}\label{eq:inequality}
    \eps\leq u_X^{\sigma}(x',x).
\end{equation}
Also, by definition of $\eps$, we have $[x]_{(\sigma,\eps)}=[x'']_{(\sigma,\eps)}$. Since $\X$ belongs to $\rep$, it also holds that $[x]_{(\sigma_0,\eps)}=[x'']_{(\sigma_0,\eps)}$, implying 
\begin{equation}\label{eq:inequality2}
    u_X^{\sigma_0}(x,x'')\leq \eps. 
\end{equation}
Since $x'$ is a conqueror of $x$ in $(X_{\sigma_0},d_X)$, we have:
\begin{equation}\label{eq:inequality3}
u_X^{\sigma_0}(x,x')\leq u_X^{\sigma_0}(x,x'').
\end{equation}
Also, since $u_X^{\sigma}\leq u_X^{\sigma_0}$, we have:
\begin{equation}\label{eq:inequality4}
u_X^{\sigma}(x,x')\leq u_X^{\sigma_0}(x,x').
\end{equation}
By concatenating inequalities (\ref{eq:inequality}), (\ref{eq:inequality2}), (\ref{eq:inequality3}), and (\ref{eq:inequality4}) we obtain: 
\[ u_X^{\sigma}(x,x')\leq u_X^{\sigma_0}(x,x'')\leq \eps\leq u_X^{\sigma}(x',x) \leq u_X^{\sigma}(x,x').
\]
The both very end sides are the same, implying that $\eps=u_X^{\sigma}(x,x')$. Since $\eps=u_X^{\sigma}(x'',x)$ and $x''$ is a conqueror of $x$ in $(X_{\sigma},d_X)$, we conclude that $x'$ is another conqueror of $x$ in  $(X_{\sigma},d_X)$, as desired.

\item $\rep \neq \eld$: It is not hard to check that the \filterMSshort{} depicted in Figure \ref{fig:elder-rule-staircode} (A) belongs to $\eld$ but not $\rep$.

\item $\eld \subsetneq \stair$: This follows from Theorem \ref{thm:elder rule2} and Example \ref{ex:non-decomposable}.

\item $\stair \subsetneq \aug$: This directly follows from Example \ref{ex:non-decomposable}.
\end{enumerate}
\end{proof}

%%%%%%%%%%%%%%%%%%%%%%%%%%%%%%%%%%%%%%%%%%%%%%%%%%%%%%%%%%%%%%%%%%%%%%%%%%%%%%%%%%%%%%%%%%%
\section{\Newbarcodes{} and graded Betti numbers}\label{sec:staircodes and Bettis} \label{sec:staircodes recover Bettis}\label{sec:bipersistence modules}

In this section we show that given an \filterMSshort{} $\X$ the graded Betti numbers of $\hzero(\ripsbi_{\bullet}(\X))$ can be easily extracted from the \newbarcodeshort{} of $\X$ (Theorem \ref{thm:recover graded Bettis}). Along the way, we obtain a characterization result for the graded Betti number of  $\hzero(\ripsbi_{\bullet}(\X))$ (Theorem \ref{thm:minimal resolution}), which is of independent interest. 

\paragraph{Computing the graded Betti numbers of $\hzero(\ripsbi_{\bullet}(\X))$ for \afilterMSshort{} $\X$.} Henceforth, for simplicity, every \filterMSshort{} $\XX$ is assumed to be \emph{generic}:  $f_X$ is injective and every pair of elements in $X$  has different distance. The case of non-generic \filterMSshort{} can be easily handled; see Remark \ref{rem:non-injective}. 
Since $\X$ is finite, it suffices to consider $\Z^2$-indexed filtration described subsequently as a substitute of  $\ripsbi_{\bullet}(\X)$ for our inductive proof of Theorem \ref{thm:recover graded Bettis}:

\begin{definition}\label{def:Z2 indexed} Consider \afilterMSshort{} $\XX$ with $X:=\{x_1,\ldots,x_n\}$ and assume that $f_X(x_1)< \ldots < f_X(x_n)$. Define $f_X^{\Z}:X\rightarrow \N$ as $x_i\mapsto i$. Define $d_X^{\Z}:X\times X\rightarrow \N$ by sending each non-trivial pair $(x_i,x_j)$ ($i\neq j$) to $\ell\in\left\{1,\ldots,\binom{n}{2}\right\}$, where $d_X(x_i,x_j)$ is the $\ell$-th smallest distance (among non-zero distance values). The restriction of $\ripsbi_\bullet(X,d_X^\Z,f_X^\Z):\R^2\rightarrow \simp$ to $\Z^2$ is the \emph{$\Z^2$-indexed Rips filtration}\footnote{$d_X^{\Z}$ does not necessarily satisfy the triangle inequality, but it does not prevent from defining $\ripsbi_{\bullet}(X,d_X^{\Z},f_X^{\Z})$.} of $\X$. Also, let $\gamma_{j}^\X$ denote the $j$-th elder-rule feature function of $(X,d_X^{\Z},f_{X}^{\Z})$ for $j=0,1,2$ in this section.
\end{definition}
For Theorem \ref{thm:minimal resolution}, we introduce relevant terminology and notation. Let $\Scal$ be the $\Z^2$-indexed Rips filtration of \afilterMSshort{} $\X$ and let $\K$ be the \emph{1-skeleton of} $\Scal$, i.e. $\K$ is another $\Z^2$-indexed filtration where $\K(\ba)$ is the 1-skeleton of $\Scal(\ba)$ for every $\ba\in \Pb$. 
\begin{itemize}[leftmargin=*]
\item Note that $\K$ is \emph{$1$-critical}: every simplex that appears in $\K$ has a unique birth index. 
 Let $e$ be an edge that appears in $\K$ whose birth index is $\bb(e)=(b_1,b_2)\in \Z^2$. We say that the edge $e$ is  \emph{negative} if the number of connected components in $\K(b_1,b_2)$ is strictly less than that of $K(b_1,b_2-1)$. Otherwise, the edge $e$ is \emph{positive}.
\item  Given a simplicial complex $K$ and $k\in \ZP$, let $C_k(K)$ be the $k$-th chain group of $K$, i.e. the $\F$-vector space freely generated by $k$-simplices in $K$. For $k\in\ZP$, let $\partial_k: C_k(K) \to C_{k-1}(K)$ be the boundary map, and  $Z_k(K):=\ker (\partial_k)$ the $k$-th cycle group of $K$.

\item Let $\K:\Z^2\rightarrow \simp$ be a filtration. For each $k\in \ZP$, let $C_k(\K):\Z^2\rightarrow \vect$ be the module defined as $C_k(\K)(\ba):=C_k(\K(\ba))$, where the internal maps $\varphi_\K(\ba,\bb)$ are the canonical inclusion maps $C_k(\K(\ba))\hookrightarrow C_k(\K(\bb))$.  In particular, if $\K$ is 1-critical, then $C_k(\K)$ is the free module whose basis elements one-to-one correspond to all the $k$-th simplices in $S$. More specifically, the birth of a simplex $\sigma\in S$ in $\K$ at $\ba\in\Z^d$ corresponds to a generator of $C_k(\K)$ at $\ba$. 
\end{itemize}

\begin{theorem}\label{thm:minimal resolution} 
Let $\K$ be the 1-skeleton of the $\Z^2$-indexed Rips filtration of \afilterMSshort{}. Let $\K^-$ be the filtration of $\K$ that is obtained by removing all positive edges in $\K$. Then, 
\begin{enumerate}[label=(\roman*)]
     \item The following sequence of persistence modules is exact:
 \begin{equation}\label{eq:minimal resolution}
     0\xrightarrow{}Z_1(\K^-) \xrightarrow{i} C_1(\K^-) \xrightarrow{\partial_1} C_0(\K^-) \xrightarrow{p} \hzero(\K)\xrightarrow{} 0, 
\end{equation}
 where $i$ is the canonical inclusion, $\partial_1$ is the boundary map, $p$ is the canonical projection.\label{item:exactness}
    \item The sequence in (\ref{eq:minimal resolution}) is a minimal free resolution of $\hzero(\K)$.\footnote{This means that $F^0= C_0(\K^-)$, $F^1=C_1(\K^-)$, $F^2=Z_1(\K^-)$ and $F^i=0$ for $i>2$ in the chain of (\ref{eq:resolution}).}\label{item:minimality}
     \end{enumerate}
\end{theorem}
We prove Theorem \ref{thm:minimal resolution} at the end of this section. For example, consider the \filterMSshort{} $\X$ in Figure \ref{fig:elder-rule-staircode} (A).  We can read off the graded Betti number of $\hzero(\ripsbi_{\bullet}(\X)):\R^2\rightarrow \vect$ from $\ripsbi_{\bullet}(\X)$.
See Figure \ref{fig:minimal resolution}. 

%%%%%%%%%%%%%%%%%
\begin{figure}
    \centering
    \includegraphics[width=\textwidth]{tame_filtration.pdf}
    \caption{(A) The \newbarcodeshort{} $\I_\X$ of $\X$ in Figure \ref{fig:elder-rule-staircode} (A). The types of corner points are indicated by circles (0-th), stars (1-st), and squares (2-nd). (B) The 1-skeleton of $\K:=\ripsbi_\bullet(\X)$. Red edges and black edges are negative and positive, respectively. The four generators of $C_0(\K)$ are located at grades $(1,0),(2,0),(3,0),(4,0)$, forming the support of the zeroth graded Betti number of $\hzero(\K)$ (marked by circles). The birth grades of four negative edges are $(2,3)$, $(3,4)$, $(4,1.5)$, and $(4,2.5)$, forming the support of the first graded Betti number (marked by stars). The unique cycle consisting solely of negative edges is $x_2x_3+x_3x_4+x_4x_2$, which is born at $(4,4)$, the unique support point of the second graded Betti number. \textbf{Observe that the locations of corner points in $\I_\X$ one-to-one correspond to the support of graded Betti numbers of $\hzero(\K)$ which illustrates that Theorem \ref{thm:recover graded Bettis} holds.} }
    \label{fig:minimal resolution}
\end{figure}

\paragraph{The \newbarcodeshort{} and the graded Betti numbers.} Next we will see that for any \filterMSshort{} $\X$, the graded Betti numbers of the zeroth homology of $\ripsbi_{\bullet}(\X)$ can be extracted from the \newbarcodeshort{} of $\X$. 

Given finite $M:\Z^2\rightarrow \vect$, the \emph{support} of the $i$-th graded Betti number $\beta_i^M$ of $M$ is defined as
$\supp(\beta_i^M):=\{\ba\in \Z^2:\beta_i^M(\ba)\neq 0\}.$
Theorem \ref{thm:minimal resolution} directly implies:

\begin{lemma}\label{cor:no intersection}Let $\K$ be the $\Z^2$-indexed Rips filtration of \afilterMSshort{} and let $M:=\hzero(\K)$. For each $i=0,1,2$, $\beta^M_i(\ba)\leq 1$, $\ba\in \Z^2$ and for every pair $i\neq j$ in $\{0,1,2\}$, $\supp(\beta_i^M)\cap \supp(\beta_j^M)=\emptyset$ 

\end{lemma}

\begin{proof}
\label{proof:no intersection}Since we concern the zeroth homology of $\K$, let us assume that $\K$ itself consists solely of vertices and edges. By Theorem \ref{thm:minimal resolution}, it suffices to show that every generator of $Z_1(\K^-)$, $C_1(\K^-)$, and $C_0(\K^-)$ is born at a different grade. In $C_0(\K^-)$, every vertex $x_i$ is born at $(i,0)$ for $i=1,\ldots,n$. Therefore, $\beta_0^M(\ba)\leq 1$ for every $\ba\in\Z^2$ and $\supp(\beta_0^M)\subset \Z\times\{0\}$. Also, by Definition \ref{def:Z2 indexed}, every generator of $C_1(\K^-)$ and $Z_1(\K^-)$ is born at different grade in $\Z\times \N$, completing the proof.
\end{proof}

%%%%%%%%%%%%%%%% 
Given any two functions $\alpha, \alpha':\Z^2\rightarrow \ZP$, we define $\alpha-\alpha':\Z^2\rightarrow \ZP$ as
\[(\alpha-\alpha')(\bx)= \max(\alpha(\bx)-\alpha'(\bx),0), \ \mbox{for $\bx\in \Z^2$}. \]

\begin{theorem}
\label{thm:recover graded Bettis}
Let $\K$ be the $\Z^2$-indexed Rips filtration of \afilterMSshort{} $\X$ and let $M:=\hzero(\K)$. Let $\beta_i^M$ be the $i$-th grade Betti number of $M$.  Then, 
\begin{equation}\label{eq:claim}
    \beta_0^M=\gamma_{0}^{\X}, \ \ \  \beta_1^M=\gamma_{1}^{\X}-\gamma_{2}^{\X}, \ \ \  \beta_2^M=\gamma_{2}^{\X}- \gamma_{1}^{\X}.
\end{equation}
\end{theorem}

  In particular, we note that the elder-rule feature functions $\gamma_j^{\X}$ are easy to compute, as one only needs to compute and aggregate the type of each corner in staircase intervals in the \newbarcodeshort{} of $\X$. Once $\gamma_j^{\X}$s are known, one can easily compute the graded Betti number of $\hzero(\ripsbi_{\bullet}(\X))$ by Theorem \ref{thm:recover graded Bettis}. 
  {See Example \ref{ex:non-decomposable} below.}  We also remark that \emph{Koszul homology formulae} \cite[Proposition 5.1]{lesnick2019computing} are in a similar form to those in (\ref{eq:claim}). However, Koszul homology formulae do not directly imply those in (\ref{eq:claim}) nor vice versa.
  
  \begin{example}[Non-interval-decomposable case]\label{ex:non-decomposable} Consider the metric space $(\{x_i\}_{i=1}^4,d_X)$ in Figure \ref{fig:elder-rule-staircode} (A). Define $h_X:\{x_i\}_{i=1}^4\rightarrow \R$ as $h_X(x_1)=1,\ h_X(x_2)=3, \ h_X(x_3)=2, h_X(x_4)=4$. For $\X:=(X,d_X,g_X)$, let $\I_{\X}:=\{I_{x_i}:i=1,2,3,4\}$ be the \newbarcodeshort{} and let $M:=\hzero(\ripsbi(\X))$ and $N:=\bigoplus_{i=1}^4 I^{I_{x_i}}$. Utilizing Theorem \ref{thm:recover graded Bettis}, it it not hard to check that $\beta_1^M\neq \beta_1^N$ and $\beta_2^M\neq \beta_2^N$ (see Figure \ref{fig:non-decomposable}). Therefore, $M\not\cong N$ and thus, by Theorem \ref{thm:strong elder rule}, $M$ is \emph{not} interval decomposable.
\end{example}

\begin{figure}
    \centering
    \includegraphics[width=\textwidth]{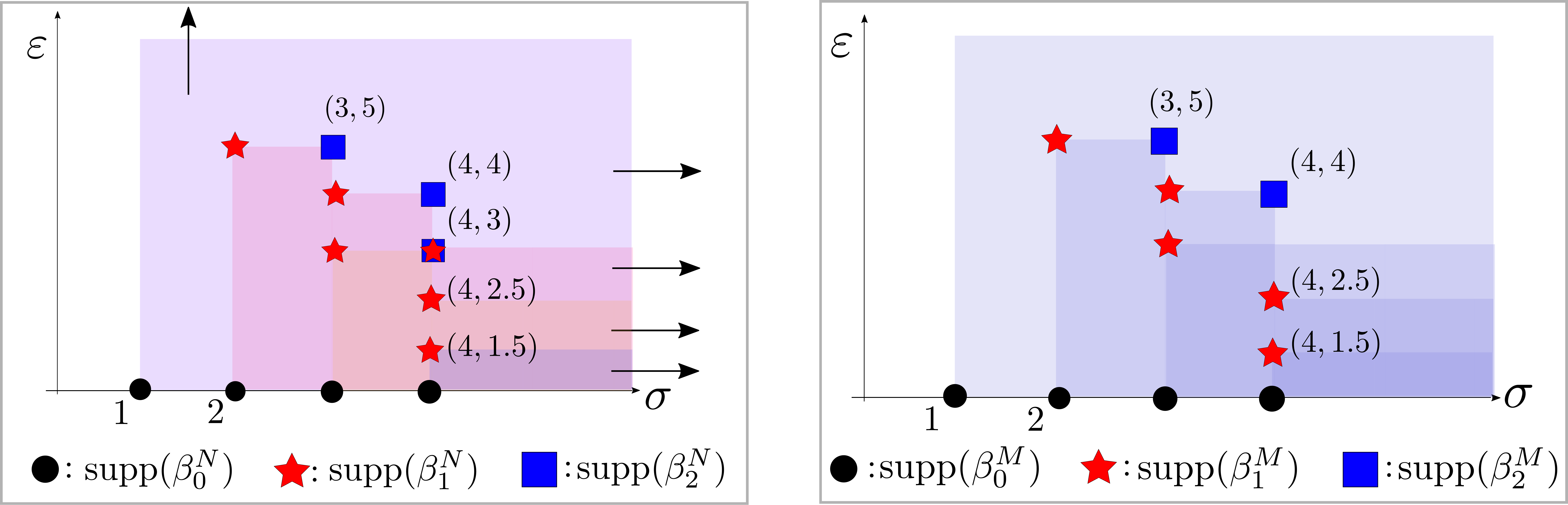}
    \caption{{The respective supports of the graded Betti numbers of $N$ (left) and $M$ (right) from Example \ref{ex:non-decomposable}. All the graded Betti numbers attain the value $1$ on their supports. The graded Betti numbers of $N$ are directly obtained by stacking the staircase intervals in the \newbarcodeshort{} of $\X$ (Remarks \ref{rem:graded betti numbers} \ref{item:graded betti numbers1} and \ref{rem:feature functions are graded Bettis}). The graded Betti numbers of $M$ are obtained by applying Theorem \ref{thm:recover graded Bettis} to the graded Betti numbers of $N$; in particular the support-points of $\beta_1^N$ and $\beta_2^N$ at $(4,3)$ are canceled out.}}
    \label{fig:non-decomposable}
\end{figure}

\begin{proof}[Proof of Theorem \ref{thm:recover graded Bettis}]
Let $\mathcal{X}:=(X,d_X,f_X)$ with $X=\{x_1,\ldots,x_n\}$, and assume that $f_X(x_1)<\ldots<f_X(x_n)$. By the construction of $\K$ and $\gamma^\X_i$, it suffices to show the equalities in (\ref{eq:claim}) hold on $\A:=\{1,2,\ldots,n\}\times\{0,1,\ldots,\binom{n}{2}\}\subset \Z^2$ ($\beta_i^M$ and $\gamma_i^\X$ vanish outside $\A$ for $i=0,1,2$). By Theorem \ref{thm:minimal resolution} and the construction of $\gamma_0^{\X}$, both of $\beta_0^M$ and $\gamma_0^{\X}$ have values 1 on $\A|_{y=0}=\{(1,0),(2,0),(3,0)\ldots,(n,0)\}$ and zero outside $\A|_{y=0}$, implying that $\beta_0^M=\gamma_0^\X$. Note that when $i=1,2$, the supports of $\beta_i^M$ and $\gamma_i^\X$ are contained in $\A|_{y>0}=\{1,2,\ldots,n\}\times\{1,\ldots,\binom{n}{2}\}$.
Using induction on $x$-coordinate of $\Z^2$, we will prove that  $\beta_1^M=\gamma_1^{\X}- \gamma_2^{\X}$ and $\beta_2^M=\gamma_2^{\X}-\gamma_1^{\X}$ \emph{on the horizontal line $\A|_{y=1}=\{1,2,\ldots,n\}\times\{1\}$}. Note that  
$\K(1,b)=\{\{x_1\}\} \ \mbox{for all $1\leq b\leq \binom{n}{2}$,}$
and thus again by Theorem \ref{thm:minimal resolution} and the construction of $\gamma_i^\X$, $i=1,2$,
\begin{equation}\label{eq:induction base}
   \ \mbox{for $1\leq b\leq \binom{n}{2}$},\  \ \beta_1^M(1,b)=\gamma_1^\X(1,b)=0, \ \mbox{and} \ \beta_2^M(1,b)=\gamma_2^\X(1,b)=0. 
\end{equation}
Specifically, we have $\beta_1^M(1,1)=\gamma_1^\X(1,1)=\gamma_1^\X(1,1)-\gamma_2^\X(1,1)$ and $\beta_2^M(1,1)=\gamma_2^\X(1,1)=\gamma_2^\X(1,1)-\gamma_1^\X(1,1)$. Fix a natural number $m>2$ and assume that $\beta_1^M(a,1)=\gamma_1^\X(a,1)-\gamma_2^\X(a,1)$ and $\beta_2^M(a,1)=\gamma_2^\X(a,1)-\gamma_1^\X(a,1)$ for $1\leq a\leq m-1$. By Theorem \ref{thm:elder rule barcode recovers the cardinality function} and Theorem \ref{thm:betti formula} in the appendix, we have: 
$\sum_{\bx\leq (m,1)}\sum_{i=0}^2(-1)^i\beta_i^M(\bx)\stackrel{(\ast)}{=}\sum_{\bx\leq (m,1)}\sum_{i=0}^2(-1)^i\gamma_i^\X(\bx).$
Since (1) $\beta_0^M=\gamma_0^\X$ on the entire $\Z^2$, and (2) $\beta_i^M,\gamma_i^\X$ vanish outside $\A$ for $i=1,2$, the induction hypothesis reduces equality $(\ast)$ to 
\[-\beta_1^M(m,1)+\beta_2^M(m,1)=-\gamma_1^\X(m,1)+\gamma_2^\X(m,1).\]
By Lemma \ref{cor:no intersection}, three cases are possible: (Case 1) $\beta_1^M(m,1)=1$ and $\beta_2^M(m,1)=0$, (Case 2) $\beta_1^M(m,1)=0$ and $\beta_2^M(m,1)=1$, or (Case 3) $\beta_1^M(m,1)=0$ and $\beta_2^M(m,1)=0$. Invoking that $\gamma_1^{\X}(m,1)$ and $\gamma_2^{\X}(m,1)$ are non-negative, in all cases, we have 
\[\beta_1^M(m,1)=\gamma_1^{\X}(m,1)- \gamma_2^{\X}(m,1), \ \ \  \beta_2^M(m,1)=\gamma_2^{\X}(m,1)- \gamma_1^{\X}(m,1),\]
completing the proof of $\beta_1^M=\gamma_1^{\X}- \gamma_2^{\X}$ and $\beta_2^M=\gamma_2^{\X}- \gamma_1^{\X}$ on $\A|_{y=1}$. We next apply the same strategy to the horizontal lines $y=2,\ldots,\ y=\binom{n}{2}$ in order, completing the proof.
\end{proof}

\begin{remark}[Theorem \ref{thm:recover graded Bettis} for non-generic cases]\label{rem:non-injective} Consider \afilterMSshort{} $\XX$ such that $\X$ is not generic. Then we pick a total order $<$ on $X$ and another total order $\prec$ on the collection of all pairs $x_i,x_j$ $(x_i\neq x_j)$ in $X$, which are compatible with $f_X$ and $d_X$, respectively. This gives rise to the injective function $f_X^{\Z}$ and the pairwise-distinct-distance $d_X^{\Z}$ on $X$ as in Definition \ref{def:Z2 indexed}. The unique \newbarcodeshort{} of $\X^{\Z}=(X,d_X^{\Z},f_Z^{\Z})$  recovers the graded Betti numbers of $\hzero(\ripsbi_{\bullet}(\X^{\Z}))|_{\Z^2}:\Z^2\rightarrow \vect$ by Theorem \ref{thm:recover graded Bettis}. \end{remark}

{Below, we will make use of Theorem \ref{thm:recover graded Bettis} in proving Theorem \ref{thm:strong elder rule}.}

\begin{proof}[Proof of Theorem \ref{thm:strong elder rule}]
Without loss of generality, let us assume that $X=\{x_1,x_2,\ldots,x_n\}$ with $f_X(x_i)=i$ for $i=1,2,\ldots,n$. Also, let $M\cong \bigoplus_{k\in K} I^{J_k}$ for some indexing set $K$. Observe that $M$ is \emph{upper-right continuous}, i.e. for each $(\sigma_0,\eps_0)\in \R^2$, there exist $e_1,e_2>0$ s.t. if $\sigma_0\leq \sigma \leq \sigma_0+e_1$ and $\eps_0\leq \eps \leq \eps_0 +e_2$, then $M_{(\sigma_0,\eps_0)}=M_{(\sigma,\eps)}$. Hence, the lower-left boundary\footnote{$(\sigma,\eps)\in \R^2$ is a lower-left boundary point of $J_k$ if  $(\sigma,\eps)$ belongs to the boundary of $J_k$ and for any $r>0$, $(\sigma-r,\eps-r)\not\in J_k$. The set of lower-left boundary points of $J_k$ is called the lower-left boundary of $J_k$.} of each $J_k$ belongs to $J_k$. Also, note that $M_{(\sigma,\eps)}\neq 0$ if and only if $(\sigma,\eps)\in U(1,0)$.

\subparagraph{Claim 1} [$\barc(M)$ consists of $n$ staircase intervals (Definition \ref{def:staircase}) and their minimal elements  are $(1,0)$,$(2,0)$,\ldots,$(n,0)$.] %
First, let us show that each interval in $\barc(M)$ is a staircase whose minimal element lies on the $\sigma$-axis. Suppose not, i.e. there exists $k_0\in K$ s.t. $J_{k_0}$ is either [not a staircase] or [a staircase whose minimum is not in the $\sigma$-axis]. Either implies that $J_{k_0}$ contains a minimal element $\ba$ in the interior of $U(0,0)$ (see Figure \ref{fig:interval shape insight in a proof}).  Then, since $M\cong \bigoplus_{k\in K} I^{J_k}$, by Remark \ref{rem:graded betti numbers} \ref{item:graded betti numbers1},
we have
\[1=\beta^{I^{J_{k_0}}}_0(\ba)\leq \sum_{k\in K}\beta^{I^{J_{k}}}_0(\ba)=\beta^M_0(\ba).\]
\begin{figure}
    \centering
    \includegraphics[width=0.6\textwidth]{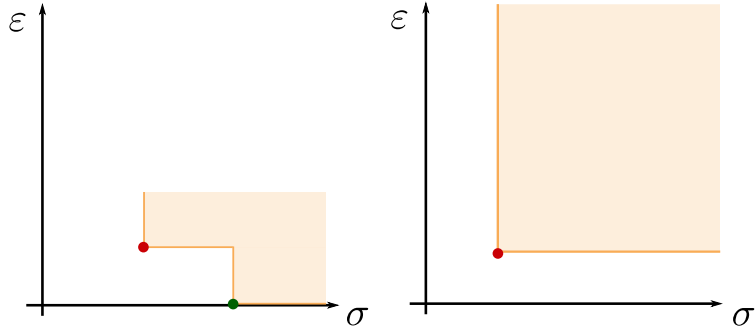}
    \caption{If an interval $J$ that is contained in the quadrant $U(1,0)$ is either [not a staircase] or [a staircase whose minimum is not in the $\sigma$-axis], then there exists a point $\ba$ in the interior of $U(0,0)$ such that  $\beta^{I^{J}}_0(\ba)=1$ (red points in the figure above).}
    \label{fig:interval shape insight in a proof}
\end{figure}
However, by Theorem \ref{thm:recover graded Bettis}, we have \[\supp(\beta^{M}_0)=\{(i,0):i=1,\ldots,n\}\not\ni \ba,\] a contradiction. Therefore, (1) each $J_k$ has its minimum element in the $\sigma$-axis, and (2) since $\beta_0^M=\sum_{k\in K}\beta^{I^{J_{k}}}_0$, by invoking Remark \ref{rem:feature functions are graded Bettis}, the minimums of $J_k$s form the set $\{(i,0):i=1,\ldots,n\}$. This implies that the indexing set $K$ contains $n$ elements, as desired.
\hfill $\square$

From now on, we denote $\barc(M)$ by $\{J_k\}_{k=1}^n$ where the minimum of $J_k$ is $(k,0)$ for each $k$. Also, let 
\begin{equation}\label{eq:set-up}
    \eps_1:=\max_{x_i,x_j\in X} d_X(x_i,x_j).
\end{equation}

\subparagraph{Claim 2}[$J_1=U(1,0)$.]  Observe that, if $\sigma\in [n,\infty)$ and $\eps\in [\eps_1,\infty)$, then 
\[\dim M_{(\sigma,\eps)}=1\ \ \ \mbox{and} \ \ \rank\ \varphi_M((1,0),(\sigma,\eps))=1.
\]
Since $\rank\ \varphi_M((1,0),(\sigma,\eps))$ is equal to the total multiplicity of elements of $\barc(M)$ which contain both $(1,0)$ and $(\sigma,\eps)$, $J_1$ must be $U(1,0)$. \hfill $\square$

Now let $f:\bigoplus_{k=1}^nI^{J_k} \rightarrow M$ be any isomorphism. For each $k$, let $1_k:=1\in (I^{J_k})_{(k,0)}(=\F)$, and let $f_{(k,0)}(1_k):=v_k\in M_{(k,0)}$. For $x_k\in X$ and $(\sigma,\eps)\in [k,\infty)\times \R_+$, let $[x_k]_{(\sigma,\eps)}$ be the zeroth homology class of $x_k$. When confusion is unlikely, we will suppress the subscript $(\sigma,\eps)$ in $[x_k]_{(\sigma,\eps)}$. 

Note that, by the definition of $M_{(k,0)}$ for each $k=1,\ldots,n$, there exist $c_{k\ell}\in \F$ for  $\ell=1,\ldots,k$ s.t.  
\begin{equation}\label{eq:basis}
\begin{aligned}
    v_1&=c_{11}[x_1]\\
    v_2&=c_{21}[x_1]+c_{22}[x_2]\\
    \vdots\\
    v_n&=c_{n1}[x_1]+\ldots+c_{nn}[x_n].    
\end{aligned}
\end{equation}
 An  $x_\ell\in X$ will be called a \emph{summand of $v_k$} if $c_{kl}\neq 0$.   Also, for each $k$, we define the function $\bv_k:U(k,0)\rightarrow \coprod_{(\sigma,\eps)\in U(k,0)}M_{(\sigma,\eps)}$ as $(\sigma,\eps)\mapsto \varphi_M((k,0),(\sigma,\eps))(v_k).$ 
 Let $\supp(\bv_k)$ be the set of $(\sigma,\eps)\in U(k,0)$ s.t. $\bv_k(\sigma,\eps)$ is nonzero in $M_{(\sigma,\eps)}$.
Since $f$ is an isomorphism, we have: 
\begin{enumerate}[label=(\roman*)]
    \item $\{\supp(\bv_k)\}_{k=1}^n=\{J_k\}_{k=1}^n$, \label{item:fact-barcode}
    \item For each $(\sigma,\eps)\in U(1,0)$, $\{\bv_k(\sigma,\eps): \sigma\in [k,\infty)\}$ is a basis of $M_{(\sigma,\eps)}$.\label{item:fact-basis}
\end{enumerate}
Now we investigate constraints on the coefficients $c_{k\ell}$.

\subparagraph{Claim 3}[For each $k$, $x_k$ is a summand of $v_k$.] By item \ref{item:fact-basis} above, the set \[B_k:=\{\bv_1(k,0),\bv_2(k,0),\ldots,\bv_k(k,0)\}\] is linearly independent in $M_{(k,0)}$. Invoking equations in (\ref{eq:basis}) and the definition of $\bv_k$, observe that if $c_{kk}=0$, then $B_k$ is linearly dependent, a contradiction. \hfill $\square$

\subparagraph{Claim 4} [For  $k\in\{2,\ldots,n\}$, $\sum_{\ell=1}^k c_{k\ell}=0$ and $v_k$ has at least two summands.] Fix $k\in\{2,\ldots,n\}$ and pick any $(\sigma,\eps)\in U(n,\eps_1)$ (see (\ref{eq:set-up})). Then, we have  $[x_{\ell_1}]_{(\sigma,\eps)}=[x_{\ell_2}]_{(\sigma,\eps)}$ for all $\ell_1,\ell_2\in \{1,\ldots,n\}$, and thus $\bv_k(\sigma,\eps)=(\sum_{\ell=1}^k c_{k\ell})\cdot [x_{k}]_{(\sigma,\eps)}$. Note that $1=\dim M_{(\sigma,\eps)}$, which is equal to the number of intervals in $\barc(M)$ that includes $(\sigma,\eps)$. Since $U(1,0)\in \barc(M)$ includes $(\sigma,\eps)$ (Claim 2), $\supp(\bv_k)$ must \emph{not} include $(\sigma,\eps)$, which implies $\sum_{\ell=1}^k c_{k\ell}$ to be $0$. This also forces $v_{k}$ to admit at least two different summands, including $x_k$ (Claim 3). \hfill $\square$

Recall that, for each $k$, $I_{x_k}$ denotes the elder-rule interval associated to $x_k$ (see (\ref{eq:barcode})).

\subparagraph{Claim 5} [For each $k$, $ I_{x_k}\subseteq \supp(\bv_k)$] By Claim 2, item \ref{item:fact-barcode} above, and Definition \ref{def:elder rule2}, we readily know $I_{x_1}=\supp(\bv_1)=U(1,0)$. Let us fix any $k\in\{2,\ldots,n\}$ and any $(\sigma,\eps)\in I_{x_{k}}$. By definition of $I_{x_k}$, $[x_k]_{(\sigma,\eps)}$ is the singleton $\{x_k\}$. Therefore, in $\bv_k(\sigma,\eps)=\sum_{\ell=1}^k c_{k\ell}[x_\ell]_{(\sigma,\eps)}$, the nontrivial term $c_{kk}[x_k]_{(\sigma,\eps)}$ cannot be combined with any other term (By Claim 3, $c_{kk}\neq 0$, and by Claim 4, there is another nonzero $c_{k\ell}$). This implies that $\bv_{k}(\sigma,\eps)\neq 0$, and in turn $(\sigma,\eps)\in \supp(\bv_k)$.\hfill$\square$

By Claim 5, we have:
\[    \hf(M)=\sum_{k=1}^n\one_{I_{x_k}} \leq \sum_{k=1}^n\one_{\supp(\bv_k)} =\hf(M).
\]
This implies that for each $k$, $\one_{I_{x_k}}=\one_{\supp(\bv_k)}$  and in turn $I_{x_k}=\supp(\bv_k)=J_k$ by item \ref{item:fact-barcode} above.
\end{proof}

\paragraph{Proof of Theorem \ref{thm:minimal resolution}.} In order to prove Theorem \ref{thm:minimal resolution}, we need the two lemmas below.

\begin{lemma}\label{lem:isomorphic without positives}Let $\K:\Z^2\rightarrow \simp$ be the 1-skeleton of the $\Z^2$-indexed Rips filtration of \afilterMSshort{}. Let $K^-$ be the filtration of $\K$ that is obtained by removing all positive edges in $\K$. Then, $\hzero(\K)\cong\hzero(\K^-)$.
\end{lemma}

\begin{proof}
\label{proof:isomorphic without positives} Observe that, for each $\ba\in\Z^2$, it holds that $\pi_0(\K(\ba))=\pi_0(\K^-(\ba))\in \subpart(X)$. Therefore, the two bipersistence treegrams $\pi_0(\K),\pi_0(\K^-):\Z^2\rightarrow \subpart(X)$ are the same. By Proposition \ref{prop:trivially isomorphic}, we have $\hzero(\K)\cong\free\circ\pi_0(\K)\cong\free\circ\pi_0(\K^-)\cong \hzero(\K^-)$.
\end{proof}

\begin{lemma}\label{prop:exactness for graphs} For any simplicial 1-complex, the following sequence is exact
\[0\xrightarrow{} Z_1(K) \xrightarrow{i} C_1(K) \xrightarrow{\partial_1} C_0(K) \xrightarrow{p} \hzero(K) \xrightarrow{} 0  ,\]
where $i$ is the canonical inclusion,  $\partial_1$ is the boundary map, and $p$ is the canonical quotient.
\end{lemma}
The proof is straightforward and thus we omit it.

For a persistence module $M$, let $IM$ denote the submodule of $M$ that is generated by the images of all linear maps $\varphi_{N}(\ba,\bb)$, with $\ba<\bb$ in $\Z^2$. We are now ready to prove Theorem \ref{thm:minimal resolution}.

\begin{proof}[Proof of Theorem \ref{thm:minimal resolution}]\label{proof:minimal resolution}Let us prove \ref{item:exactness}. 
By Lemma \ref{lem:isomorphic without positives}, $\hzero(\K^-)$ is isomorphic to $\hzero(\K)$ and thus it suffices to show the exactness of 
\[  0\xrightarrow{}Z_1(\K^-) \xrightarrow{i} C_1(\K^-) \xrightarrow{\partial_1} C_0(\K^-) \xrightarrow{p} \hzero(\K^-)\xrightarrow{} 0.\]
At each grade $\ba\in \Z^2$, we have the sequence of vector spaces and linear maps:
\[  0\xrightarrow{}Z_1(\K^-_{\ba}) \xrightarrow{i_{\ba}} C_1(\K^-_{\ba}) \xrightarrow{(\partial_1)_{\ba}} C_0(\K^-_{\ba}) \xrightarrow{p_{\ba}} \hzero(\K^-_{\ba})\xrightarrow{} 0,\]
which is exact by Lemma \ref{prop:exactness for graphs}.

Next, we prove \ref{item:minimality}. In the following proof, we assume the ground field $\F$ is $\Z_2$ for the sake of simplicity. We need to show that (a) $C_0(\K^-)$, $C_1(\K^-)$, and $Z_1(\K^-)$ are free modules, and that (b) the sequence in (\ref{eq:minimal resolution}) satisfies the minimality condition. Let us prove (a). By definition, it is clear that $C_0(\K^-)$ and $C_1(\K^-)$ are free . Also, $Z_1(\K^1)$, the kernel of $\partial_1$, is free by \cite[Section 6]{chacholski2017combinatorial}.\footnote{The authors of \cite{chacholski2017combinatorial} observe that for any two free modules $M,N:\Z^2\rightarrow \vect$, the kernel of any natural transformation $f:M\rightarrow N$ is a free module.} Let us check (b). We show that the image of $C_1(\K^-)$ via $\partial_1$ is contained in $IC_0(\K^-)$. It suffices to show that every generator of $C_1(\K^-)$ is mapped into $IC_0(\K^-)$. Pick any edge $x_ix_j$ ($i<j$) that appears in $\K^-$. Then, in the filtration $\K^-$, $x_ix_j$ is born at $(j,d_X^\Z(x_i,x_j))=:\ba$, whereas the vertices $x_i$ and $x_j$ are born at $(i,0)$ and $(j,0)$, respectively. Note that  $(i,0)<(j,0)<\ba$ $\mbox{in $\Z^2$}$. Therefore, $\partial_1|_{\ba}(x_ix_j)=x_i+x_j\in IC_0(\K^-)_{\ba}$.

Let us show that the image of $Z_1(\K^-)$ is contained in $IC_1(\K^-)$. To this end, it suffices to show that every generator of $Z_1(\K^-)$ is mapped into $IC_1(\K^-)$. Suppose that a cycle \[c\stackrel{(\ast)}{=}x_{i_1}x_{j_1}+x_{i_2}x_{j_2}+\ldots+x_{i_l}x_{j_l}
\] 
is born at grade $\bb=(b_1,b_2)\in \Z^2$ in $\K^-$. Consider the vertical restriction $\K^-|_{x=b_1}:\{b_1\}\times \Z\rightarrow \simp$. By the construction of $\K$ (Definition \ref{def:Z2 indexed}), 
at most one simplex can be added at each grade in $\K^-|_{x=b_1}$ as index increases. Thus, there exists a unique edge on the RHS of equality ($\ast$) which appears at $\bb$ in $\K^-|_{x=b_1}$ (if no edge is born at $\bb$, then $c$ cannot be born at $\bb$, contradicting the assumption). Without loss of generality, let $x_{i_l}x_{j_l}$ be such edge. We will show that $x_{i_l}x_{j_l}$ is not born at $\bb$ in $\K^-$ by contradiction. Suppose that $x_{i_l}x_{j_l}$ is born at $\bb$ in $\K^-$. This implies that  $x_{i_l}x_{j_l}$ is born at $\bb$ in $\K(\supset \K^-)$, generating the cycle $c$. This means that $x_{i_l}x_{j_l}$ is positive, contradicting the fact that $x_{i_l}x_{j_l}$ is taken in the filtration $\K^-$ whose edges are negative in $\K$. Therefore, every edge on the RHS of equality $(\ast)$ is born at grades strictly smaller than $\ba$. This implies that $\partial_1|_{\bb}(c)\in IC_1(\K^-)_{\bb}$, as desired. \end{proof}

%%%%%%%%%%%%%%%%%%%%%%%%%%%%%%%%%%%%%%%%%%%%%%%%%%%%%%%%%%%%%%%%%
\section{Computation and Algorithms}\label{sec:algorithm}

\subsection{Algorithm}
\begin{theorem}\label{thm:computation}
Let $(X, d_X, f_X)$ be a finite \filterMSshort{} with $n = |X|$. 

(a) We can compute the \newbarcodeshort{} $I_{\X}=\lmulti I_{x}:x\in X\rmulti$ in $O(n^2 \log n)$ time. 
If $X\subset \mathbb{R}^d$ for a fixed $d$ and $d_X$ the Euclidean distance, the time can be improved to $O(n^2 \alpha(n))$, where $\alpha(n)$ is the inverse Ackermann function. 

(b) Each $I_x \in I_{\X}$ has complexity $O(n)$.  Given $I_{\X}$, we can compute zeroth fibered barcode $B^L$ for any line $L$ with positive slope in $O(|B^L| \log n)$ time where $|B^L|$ is the size of $B^L$. 

(c) Given $I_{\X}$, we can compute the zeroth graded Betti numbers in $O(n^2)$ time. 

\end{theorem}

We sketch the proof of the above theorem in the remainder of this section, with missing details in Appendix \ref{algo_appendix}. 

Consider a function value $\sigma \in \mathbb{R}$, and recall that $X_\sigma$ consists of all points in $X$ with $f_X$ value at most $\sigma$. Let $\K_\sigma = \rips_{\bullet}(X_\sigma,d_X)$ denote the Rips filtration of $(X_\sigma, d_X)$ (recall Remark \ref{rem:treegrams arise from filtrations}). The corresponding 1-parameter treegram (dendrogram) is $\theta_\sigma:= \pi_0(\K_\sigma)$. 
On the other hand, for any $\sigma$, we can consider the \emph{complete weighted graph} $G_\sigma = (V_\sigma = X_\sigma, E_\sigma)$ with edge weight $w(x,x') = d_X(x,x')$ for any $x, x'\in X_\sigma$. It is folklore that the treegram $\theta_\sigma$ can be computed from the minimum spanning tree (MST) $T_\sigma$ of $G_\sigma$. 

Assume all points in $X$ are ordered $x_1, x_2, \ldots, x_n$ such that $f_X(x_i) \le f_X(x_j)$ whenever $i< j$, and set $\sigma_i = f(x_i)$ for $i\in [1,n]$. Note that as $\sigma$ varies, $X_\sigma$ only changes at $\sigma_i$. 
For simplicity, we set $\theta_i:= \theta_{\sigma_i} = \pi_0(\K_{\sigma_i})$, 
$G_i:= G_{\sigma_i}$ and $T_i:= MST(G_i)$ is the minimum spanning tree (MST) for the weighted graph $G_i$. 
Our algorithm depends on the following lemma, the proof of which is in Appendix \ref{app:dec_treegram}. 
\begin{lemma}\label{lem:recoverStaircode}
A decorated \newbarcodeshort{} for the finite \filterMSshort{} $(X, d_X, f_X)$ can be computed from the collection of treegrams $\{ \theta_i , i\in [1, n] \}$ in $O(n^2)$ time. 
\end{lemma}

In light of the above result, the algorithm to compute \newbarcodeshort{} is rather simple: 
\begin{description}
\item[(Step 1):] We start with $T_0 =$ empty tree. At the $i$-th iteration, 
\begin{description}
\item[(Step 1-a)] 
we update $T_{i-1}$ (already computed) to obtain $T_i$; and 
\item[(Step 1-b)] compute $\theta_i$ from $T_i$ and $\theta_{i-1}$. 
\end{description} 
\item[(Step 2):] We use the approach described in the proof of Lemma \ref{lem:recoverStaircode} to compute the \newbarcodeshort{} in $O(n^2)$ time. 
\end{description}

For (Step 1-a), note that $G_i$ is obtained by inserting vertex $x_i$, as well as all $i-1$ edges between $(x_i, x_j)$, $j\in [1, i-1]$, into graph $G_{i-1}$. By \cite{chin1978algorithms}, one can update the minimum spanning tree $T_{i-1}$ of $G_{i-1}$ to obtain the  MST $T_i$ of $G_i$ in $O(n)$ time. 

For (Step 1-b), once all $i-1$ edges spanning $i$ vertices in $T_i$ are sorted, then we can easily build the treegram $\theta_i$ in $O(i\alpha (i)) = O(n \alpha(n))$ time, by using union-find data structure (see Figure \ref{fig:decorated_dendrogram} in Appendix \ref{app:dec_treegram}). Sorting edges in $T_i$ takes $O(i \log i) = O(n\log n)$ time. 
Hence the total time spent on (Step 1-b) for all $i\in[1, n]$ is $O(n^2 \log n)$. 

We remark that knowing the order of all edges in $T_{i-1}$ may not help, as compared to $T_{i-1}$, $T_i$ may have $\Omega(i)$ different edges newly introduced, and these new edges still need to be sorted. 
Nevertheless, we show in Appendix \ref{app:algo_proof} that if $X \subset \mathbb{R}^d$ for a fixed dimension $d$, then each $T_i$ will only have constant number of different edges compared to $T_{i-1}$, and we can sort all edges in $T_i$ in $O(n)$ time by inserting the new edges to the sorted list of edges in $T_{i-1}$. Hence $\theta_i$ can be computed in $O(n\alpha(n)) + O(n) = O(n\alpha(n))$ time for this case. 

Putting everything together, Theorem \ref{thm:computation} (a) follows. See Appendix \ref{app:algo_proof} for the proofs of (b) and (c). 

\begin{table}[]
\centering
\caption{Complexity Comparison with \cite{lesnick2015interactive} and \cite{lesnick2019computing} for computing the fibered barcode and graded Betti number of $\hzero(\K)$, where $\K$ is the 1-skeleton $\Z^2$-indexed of Rips bifiltration of \afilterMSshort{} of  $n$ points.  $\BML$ is the cardinality of the fibered barcode for query line $L$ of positive slope.}
\begin{tabular}{@{}lccc@{}}
\hline
 & Our Algorithm & RIVET \cite{lesnick2015interactive} & Graded Betti number \cite{lesnick2019computing} \\ \hline
Size of descriptor & $O(n^2)$ & $O(n^6) \sim O(n^{8})$ & $\Omega(n^2)$ \\
Fibered barcodes query time & $O(\BML \log n)$ & $O(\log n + \BML )$ & - \\
Computation time & $O(n^2 \log n)$ & $O(n^8) \sim O(n^9)$ & $\Omega(n^3)$ \\ \hline
\label{compare}
\end{tabular}
\end{table}

\subsection{Comparison with other algorithms}
\label{appendix:comparasion}

Let $\mathcal{\K}$ be the 1-skeleton of the $\Z^2$-indexed Rips filtration of \afilterMSshort{} $\XX$, where $\abs{X}=n$. Let $M := \hzero(\mathcal{\K})$.

\paragraph{Comparison with \cite{lesnick2015interactive}.}  Let $\kappa:=\kappa_x\kappa_y$, where $\kappa_x$ and $\kappa_y$ are the number of $x$ and $y$ coordinates in $\supp(\beta_0^M)\cup \supp(\beta_1^M)$, respectively. In our case $\kappa_x=n$ and \[\kappa_y=\mbox{(the number of negative edges in $\K$)},\] which is between $O(n)$ and $O(n^2)$. Let $m$ be the number of simplices in $\K$, which is $O(n^2)$.

 From the filtration $\K$, RIVET computes a certain data structure $\mathcal{A}^{\bullet}(M)$ of size $O\left(m \kappa^{2}\right)$ in $O\left(m^{3} \kappa+(m+\log \kappa) \kappa^{2}\right)$ time and $O\left(m^{2}+m \kappa^{2}\right)$ memory. This $\mathcal{A}^{\bullet}(M)$ allows efficient query about the fibered barcode of $M$ in $O(\log \kappa + \BML)$ where $\BML$ is the size of the fibered barcode $\barc (M|_L)$ for a positive slope line $L\in\Lcal$.

See Table \ref{compare} for the comparison of computational complexity between RIVET and our method.

\paragraph{Comparison with \cite{lesnick2019computing}.} The algorithm in \cite{lesnick2019computing} takes as input a short chain complex of free modules $F^{2} \stackrel{\partial^{2}}{\longrightarrow} F^{1} \stackrel{\partial^{1}}{\longrightarrow} F^{0}$ such that $M \cong \operatorname{ker} \partial^{1} / \operatorname{im} \partial^{2}$ and outputs a minimal presentation of a 2-parameter persistence module $M$, from which the graded Betti numbers of $M$ are readily computed. It runs in time $O\left(\sum_{i}\left|F^{i}\right|^{3}\right)$ and requires $O\left(\sum_{i}\left|F^{i}\right|^{2}\right)$ memory, where $|F^i|$ denotes the size of a basis of $F^i$. In our setting, we readily have  $|F^{0}| = 0, |F^{1}| = n$, $\abs{F^2}$=(the number of negative edges in $\K$) which is between $O(n)$ and $O(n^2)$. Therefore, in order to obtain the graded Betti numbers via the method in \cite{lesnick2019computing}, it takes at least $\Omega(n^3) $ time and $\Omega(n^2)$ memory.

\section{Discussion}\label{sec:conclusion}

Some open questions and conjectures follow: 

\begin{enumerate}[wide, labelwidth=!, labelindent=0pt]
    \item \textbf{Barcodes and \newbarcodes.} 
    {(1) Let $\XX$ be \afilterMSshort{}. If $x\in X$ has a constant conqueror, is the interval module supported by $I_{x}$ in (\ref{eq:barcode}) a summand of $\hzero(\ripsbi(\X))$? (2) By virtue of Theorem \ref{thm:strong elder rule}, if $\hzero(\ripsbi(\X))$ is interval decomposable, then the \newbarcodeshort{} is identical to the \emph{generalized persistence diagram} of $\hzero(\ripsbi(\X))$ \cite{kim2018rank}. In general: What is the relation between the \newbarcodeshort{} and the generalized persistence diagram?} 
    \item \textbf{Extension to $d$-\filterMSshort{}s.} Can we generalize our results to the setting of more than two parameters? Namely, for $d$-\filterMSshort{}s $\X^d:=(X,d_X,f_1,f_2,\ldots,f_{d})$, $f_i:X\rightarrow \R$, $i=1,\ldots,d$, can we recover the zeroth homological information of the $d+1$-parameter filtration induced by $\X^d$ by devising ``\anewbarcode{}'' of $\X^d$? Note that, under the assumption the set $\{\left(f_i(x)\right)_{i=1}^d\in \R^d: x\in X\}$ is totally ordered in the poset $\R^d$, a straightforward generalization of the elder-rule staircode is conceivable. However, without this strict assumption, it is not very clear how  \newbarcode{}s should be defined.
    \item \textbf{Extension to higher-order homology.} The  ambiguity mentioned  in the previous paragraph also arises when trying to devise an ``elder-rule-staircode'' for  higher-order homology of a multiparameter filtration; namely, when $k\geq 1$, the birth indices of $k$-cycles are not necessarily totally ordered in the multiparameter setting, and thus determining which cycle is older than another is not clear in general.
    \item \textbf{Metrics and stability.} 
    Recall that the collection $E(\X)$ of all possible \newbarcodeshort{}s of an \filterMSshort{} $\X$ is an invariant of $\X$ (the paragraph after Example \ref{ex:constant}). One possible metric between two collections of \newbarcodeshort{}s is the Hausdorff distance $d_{\mathrm{H}}^b$ in the metric space of barcodes over $\R^2$ with the generalized bottleneck distance $d_b$  \cite{botnan2018algebraic}. On the other hand, there exists a metric $d_{\mathrm{GH}}^1$ which measures the difference between \filterMSsshort{} \cite{carlsson2010multiparameter} (see also \cite{chazal2009gromov}) and let $d_{\mathrm{I}}$ be the interleaving distance between 2-parameter persistence modules \cite{lesnick}. Are there constants $\alpha,\beta>0$ such that for all \filterMSsshort{} $\X$ and $\Y$, the inequalities below hold? 
    \label{item:metrics and stability}
\[\alpha\cdot d_{\mathrm{I}} \left(\hzero\left(\ripsbi_{\bullet}(\X)\right),\hzero\left(\ripsbi_{\bullet}(\Y)\right) \right) \leq d_{\mathrm{H}}^b(E(\X),E(\Y))\leq \beta\cdot d_{\mathrm{GH}}^1(\X,\Y).
\]
\item \textbf{Completeness.} Recall that the collection $E(\X)$ of all the \newbarcode{}s of \afilterMSshort{} $\X$ is not a complete invariant (the paragraph after Example \ref{ex:constant}). How faithful is this collection in general? Is there any class of \filterMSshort{}s $\X$ such that $E(\X)$ completely characterizes $\X$? \label{item:completeness}
\end{enumerate}

\appendix

%%%%%%%%%%%%%%%%%%%%%%%%%%%%%%%%%%%%%%%%%%%%%%%%%%%%%%%%%%%%%%%%%%%%%%%%%%%%%%%%%%%%%%%%%%%%%%%%%%%%%%%%%%%%%%%%%%%%%%
\section{Missing details from Section \ref{sec:newbarcode}}\label{sec:proofs}

\begin{proof}[Proof of Proposition \ref{prop:staircode is an interval}]
\label{proof:staircode is an interval}  Let $x\in X$ be the point which achieves the minimum of $f_X$. Then $I_x=\{\sigeps\in \R^2:(f_X(x),0)\leq \sigeps\}$, the closed quadrant whose lower-left corner is $(f_X(x),0)$. Let $y\in X$ be a point which does not achieve the minimum of $f_X$. Define $u_y:\R\rightarrow \RP$ by sending $\sigma\in \R$ to the minimum $\eps\in\RP$ for which there exists $z\in X$ with $f_X(z)<f_X(y)$ such that $y$ belongs to the same block with $z$ in the partition $\pi_0(\rips_{\eps}(X_{\sigma},d_X))$ (see the paragraph after Definition \ref{def:treegram}). It is clear that $u_y$ is non-increasing. Also, since $X$ is finite, $u_y$ is piecewise constant. By observing $I_y=\{\sigeps\in\R^2: \sigma\in [f_X(y),0)\ \mbox{and}\  \eps\in [0,u_y(\sigma))\}$, we complete the proof.
\end{proof}

We precisely define the $j$-th type corner points of staircase intervals depicted in Figure \ref{fig:corner types}:

\begin{definition}[Types of corner points]\label{def:types of corners} Let $I$ be a staircase interval of $\R^2$. Fix $\ba\in \R^2$. This $\ba$ is a \emph{0-th type corner point} of $I$ if \[\one_I(\ba)=1, \ \ \  \lim_{\eps\rightarrow0+}\one_I(\ba-(\eps,0))=\lim_{\eps\rightarrow0+}\one_I(\ba-(0,\eps))=\lim_{\eps\rightarrow0+}\one_I(\ba-(\eps,\eps))=0.\] The point $\ba$ is a \emph{1-st type corner point} of $I$ if \[ \one_I(\ba)- \lim_{\eps\rightarrow0+}\one_I(\ba-(\eps,0))-\lim_{\eps\rightarrow0+}\one_I(\ba-(0,\eps))+\lim_{\eps\rightarrow0+}\one_I(\ba-(\eps,\eps))=-1.\] The point $\ba$ is a \emph{2-nd type corner point} of $I$ if \[\one_I(\ba)=0, \ \ \  \lim_{\eps\rightarrow0+}\one_I(\ba-(\eps,0))=\lim_{\eps\rightarrow0+}\one_I(\ba-(0,\eps))=\lim_{\eps\rightarrow0+}\one_I(\ba-(\eps,\eps))=1.\]
\end{definition}
We remark that Definition \ref{def:types of corners} is closely related to the \emph{differential} of an interval introduced in \cite{dey2019generalized}.

\section{Missing details from Section \ref{sec:decorated staircodes and treegrams}}
In order to show that Definition \ref{def:elder rule barcode} is well-defined, it suffices to show:

\begin{proposition}[Elder-rule-barcode is well-defined]\label{prop:elder-rule-barcode is well-defined for treegram}
Let $\theta_X:\R\rightarrow \subpart(X)$ be a treegram over $X$ and suppose that there exist different $y,z\in X$ with $b(y)=b(z)$. Consider two orders $<_1,<_2$ which are the same except for the pair $y,z$, i.e. $y<_1 z$ and $z <_2 y$.  Then,
$\lmulti \left[b(x),{d^{<_1}}(x)\right):x\in X\rmulti = \lmulti \left[b(x),{d^{<_2}}(x)\right):x\in X\rmulti.$
\end{proposition}

\begin{proof}
\label{proof:elder-rule-barcode is well-defined for treegram}  For $x\in X$ different from $y$ and $z$, it is clear that $\left[b(x),{d^{<_1}}(x)\right)=\left[b(x),{d^{<_2}}(x)\right)$. Hence, letting $b:=b(y)=b(z)$, it suffices to show that \[\lmulti \left[b,d^{<_1}(y)\right),\left[b,d^{<_1}(z)\right)\rmulti =\lmulti \left[b,d^{<_2}(y)\right),\left[b,d^{<_2}(z)\right)\rmulti,
\] or equivalently $\lmulti d^{<_1}(y),d^{<_1}(z)\rmulti =\lmulti d^{<_2}(y),d^{<_2}(z)\rmulti$. Assume that $y$ and $z$ merge at $\eps=r_0$ in $\theta_X$. Since $<_1$ and $<_2$ are the same except for the pair $y,z$, we use $<$ to denote both $<_1$ and $<_2$ \emph{when we compare $y,z$ with the other elements in $X$}. In the treegram $\theta_X$, there are only two possible cases: (Case 1) An element $w\in X$ with $w<y,z$ merges with the block containing both $y,z$ at $\eps=r_1\geq r_0$. Then, $d^{<_1}(y)=r_1$ and $d^{<_1}(z)=r_0$, whereas $d^{<_2}(y)=r_0$ and $d^{<_2}(z)=r_1$.  (Case 2) assume that there are $w_1<y$ and $w_2<z$ such that $w_1$ and $y$ merge at $\eps=r_2\leq r_0$ and $w_2$ and $z$ merge at $\eps=r_3\leq r_0$ (it is possible that $w_1=w_2$). Then, $d^{<_1}(y)=d^{<_2}(y)=r_2$ and $d^{<_1}(z)=d^{<_2}(z)=r_3$, completing the proof.
\end{proof}

\section{Missing details from Section \ref{sec:algorithm}} \label{algo_appendix}
\subsection{Proofs of Theorem \ref{thm:computation}} \label{app:algo_proof}

We first present a lemma needed for the proof of Theorem \ref{thm:computation} (a). For simplicity, we assume that all distances between points in $X$ (and thus edge weights in $G_i$s) are distinct. If this is not the case, we only need to fix a total order compatible with all distances for the algorithm to work in the same way. 

\begin{lemma}
\label{oldnever}
Given a graph $G = (V, E)$ with distinct edge weights, if $e\in E$ is the largest edge of a cycle $C$ in $G$, then $e$ will not appear in the MST of $G$.
\end{lemma}

\begin{proof}
 Let us denote $e$ as the largest edge in the cycle $C$ of size $k+1$ where $C$ consists of edges ${e, e_1, e_2, ... , e_k}$. Also denote the MST of $G$ as $T$. From $C$ and $T$, We will give a way to construct new cycle $C'$ where all edges except $e$ belong to $T$.

Since $T$ is MST, for any $i\in\{1, 2,..., k\}$, if $e_i$ does not belong to $T$, adding $e_i$ will form a cycle $C_i$ where $e_i$ in the largest edge and the only non-MST edge in $C_i$. 

Construct new cycle $C'=C + \Sigma_{i \in \{ j | e_j \notin T\} } C_i $ where the addition is performed on $\mathbb{F}_2$.
 Every time we add $C_i$, it will cancel out $e_i$. Since we did for all non-MST edges, the resulting cycle $C'$ will consist of all MST edges plus $e$.

 We argue that $e$ is also the largest edge in $C'$. This holds because every time we add $C_i$, we knew $e_i$ is the largest edge in $C_i$, and because $|w(e)| \geq |w(e_i)|$ where $w$ is weight function on edges, we knew $e$ is also the largest edge in $C'$. By the property of MST (any non-MST edge is the largest edge in the cycle created by adding itself to MST.), we conclude that $e$ is a non-MST edge.
\end{proof}

\begin{figure}
\centering
\includegraphics[scale=.5]{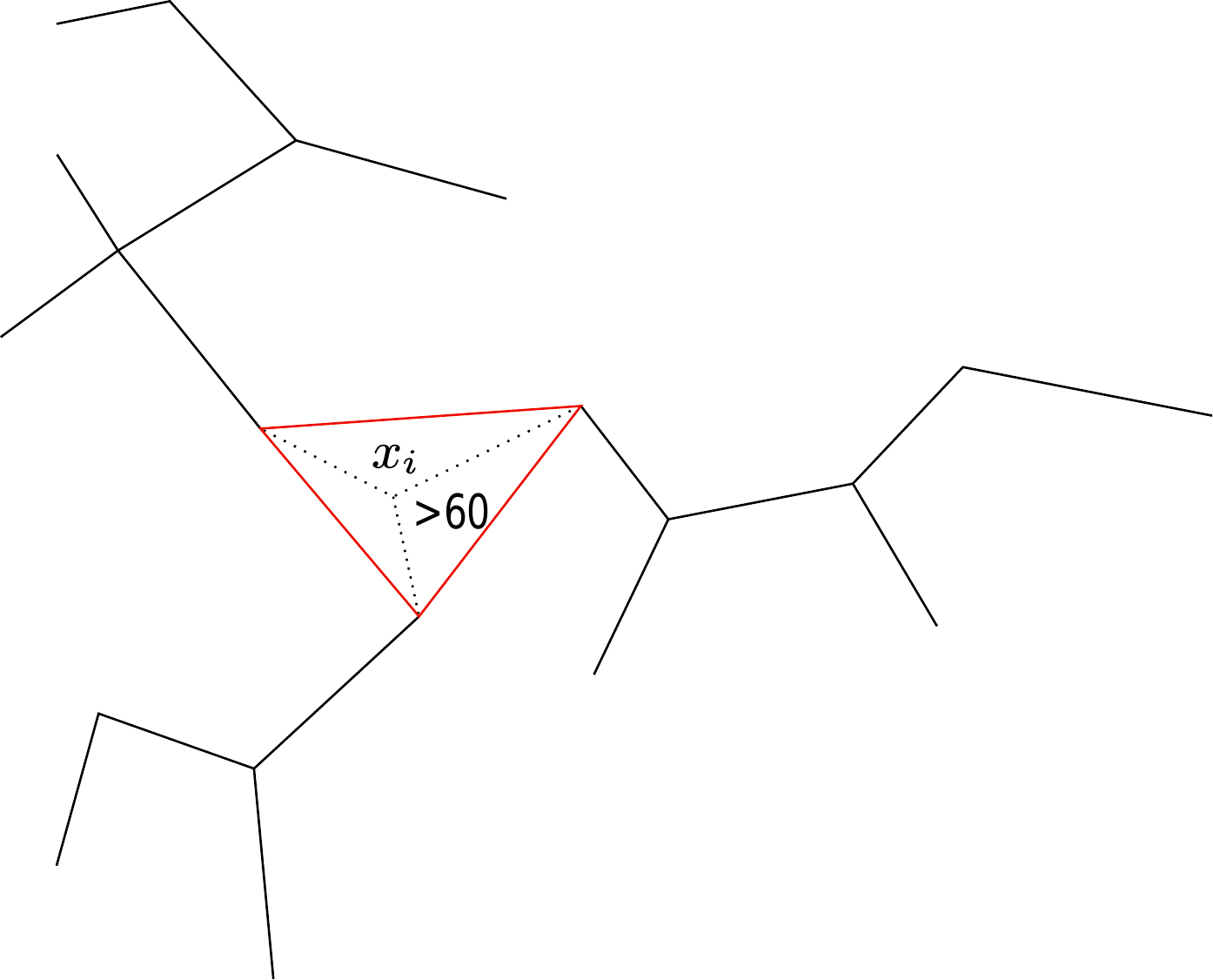}
\caption{Illustration of packing argument in Theorem \ref{thm:computation}. $x_i$ is the new vertex. Dashed edges are new edges entering $T_i$. Red edges are non-MST edges and therefore by the property of non-MST edges, angle corresponding to red edges must be at least $\frac{\pi}{3}$.}
\label{fig:pack}
\end{figure}

The following lemma, combined with the argument in the main text, will establish the time complexity of the algorithm to compute \newbarcodeshort{} for the case when $X$ is from a fixed dimensional Euclidean space $\mathbb{R}^d$. 

\begin{lemma}\label{lem:Euclideancase}
Let $T_{i-1}$ and $T_i$ be the MST of $G_{i-1}$ and $G_i$ as defined in the algorithm. For fixed dimensional $\mathbb{R}^d$ and $d_X$ to be Euclidean distance, the number of edges in $T_i \setminus T_{i-1}$ is $O(1)$ (depending on $d$). 
\end{lemma}

\begin{proof}
Recall that $G_i$ is obtained by adding a new vertex $x_i$ and edges incident to $x_i$. 
First, note that by Lemma \ref{oldnever}, edges in $T_i = MST(G_i)$ are either from $T_{i-1} = MST(G_{i-1})$, or new edges just inserted. That is, no edge in $G_{i-1} \setminus T_{i-1}$ will contribute to $T_i$: This is such an edge will be the largest-weight edge of some cycle in $G_i$. 

We now prove that for the case where $X \subset \mathbb{R}^d$ only $O(1)$ (where the big-O hides terms depending on $d$) new edges (incident to $x_i$) can be in $T_i$. 

In particular, comparing $T_{i-1}$ and $T_i$, there are only two types of edges that are subject to change: 1) edges that are in $T_{i-1}$ but will leave $T_{i-1}$ and 2) edges incident to $x_i$ and will enter the new $T_i$.

Assume there are $k$ edges that will leave $T_{i-1}$. By deleting them, the original $T_{i-1}$ is decomposed into $k+1$ small trees. There must be $k+1$ edges incident to $x_i$ entering $T_i$. We denote those $k+1$ edges as $\mathcal{E}_{new, i} = \{x_ix_{i_1}, x_ix_{i_2}, ..., x_ix_{i_{k+1}}\}$

Pick any two nodes $a,b$ from $\mathcal{E}_{new, i} = \{x_{i_1}, x_{i_2}, ..., x_{i_{k+1}}\}$, they will from a triangle with $x_i$. By property of MST, edge $ab$ in triangle $\bigtriangleup_{abx_i}$ is the longest edge as $ab\notin T_i$ while $x_ia, x_ib \in T_i$. 
By elementary Euclidean geometry, it can be shown that angle $\sphericalangle ax_ib$ must be no less than $\frac{\pi}{3}$, and this holds for {\bf every pair} of nodes from $\mathcal{E}_{new, i} =\{x_{i_1}, x_{i_2}, ..., x_{i_{k+1}}\}$. Now by a packing argument, we can show that there can be $O(C^d)$ such well-separated points around $x_i$ in $\mathbb{R}^d$ for some constant $C$. 

Indeed, consider the unit sphere $S$ around $x_i$ in $\mathbb{R}^d$, and let $y_j$ be the intersection of the ray starting at $x_i$ and passing through $x_{i_j}$ with $S$. 
The previous paragraph establishes that the angle $\sphericalangle y_j x_i y_{j'} \ge \pi/3$ for any $j \neq j' \in [1, k+1]$. It then follows that the geodesic distance between $y_j$ and $y_{j'}$ on $S$ is at least $\pi/3$. In other words, geodesic balls of radius $\pi/6$ centered at $y_j$'s for $j\in [1, k+1]$ have to be all disjoint. The number of such balls (and thus $k+1$) is at most $Area(S) / B$ where $Area(S)$ stands for the surface volume of unit $d$-sphere in $\mathbb{R}^d$, while $B$ is the volume of a ($d-1$)-ball of radius $sin \frac{\pi}{6} = \frac{1}{2}$. Hence there exists some constant $C > 1$ such that $k = O(C^d)$. This proves the lemma. 
\end{proof}

We now present proofs for part (b) and (c) of Theorem \ref{thm:computation}.
\begin{lemma}
The size of \newbarcodeshort{} is $O(n^2)$.
\end{lemma}

\begin{proof}
We claim that for every $x \in \mathcal{X}$, the size of $I_x$ is $O(n)$ and the lemma will then follow. This holds because every $I_x$ has a staircase shape, and the $x$-coordinates of corner points can only be one of the values $f_X(x_i)$ for some $x_i \in X$.
\end{proof}

\begin{figure}
\centering
\includegraphics[scale=.2]{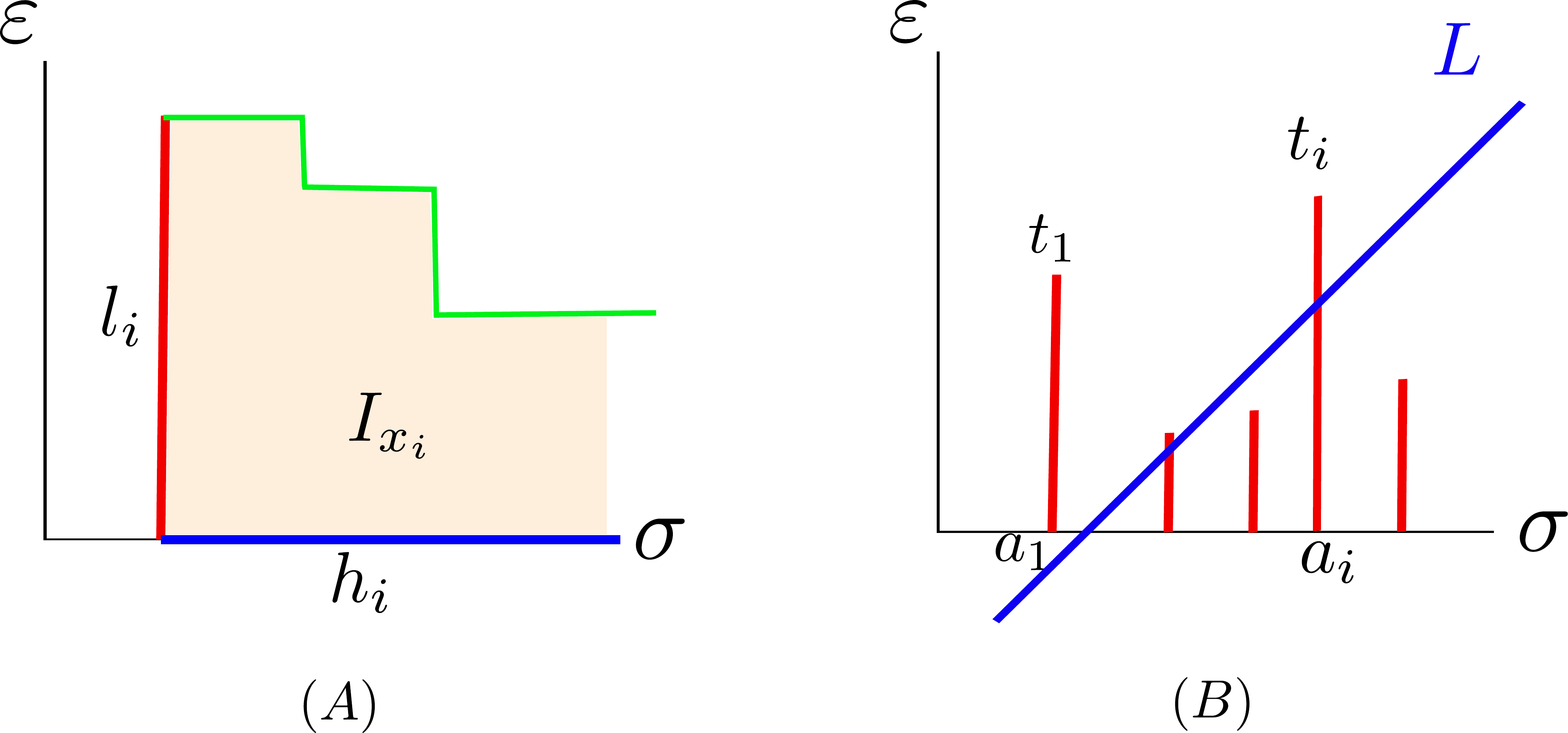}
\caption{(A) An illustration of $I_{x_i}$ where its lower envelope $l_i$ (vertical line segment) and $h_i$ (horizontal ray) are colored red and blue. (B) An illustration of Case 2 in Lemma \ref{lem:siez_of_staircode}}
\label{fig:query_time}
\end{figure}

\begin{lemma}
\label{lem:siez_of_staircode} 
Given $I_{\X}$, after $O(n^2\log n)$ time preprocessing, we can build a data structure of size $O(n^2)$ so that, given any line $L$ with positive slope, the zeroth fibered barcode $B^L$ w.r.t. $L$ can be computed in $O(|B^L| \log n)$ time where $|B^L|$ is the size of fibered barcode. 
\end{lemma}

\begin{proof}
First, given an $I_x$, recall that it has a staircase shape: see Figure \ref{fig:query_time}. In particular, its lower envelop consists of one vertical and one horizontal segment. Its upper envelope $U$ is the graph of a piecewise constant non-decreasing function in the plane consisting of $O(n)$ horizontal and vertical line segments. 
Given a line $L$ with positive slope, its intersection with the lower envelop of $I_x$ thus takes only $O(1)$ time. The upper envelope can only intersect with $L$ at most one point, either within some horizontal segment of $U$ or within a vertical segment of $U$. To identify this intersection point, we simply binary search twice: once among all horizontal segments, and once among all vertical segments, in $O(\log n)$ time. 

Next, we show that we can avoid checking all $n$ number of $I_x$s. Instead, we will compute only the set $\widehat{\I}_L$ of $I_x$s that will intersect $L$: Note that there are $k = |B^L|$ number of such \newstair{}s. In what follows, we describe how to preprocess all \newstair{}s so that this set  $\widehat{\I}_L$ can be reported in $O(\log n + k)$ time. 

Specifically, for any $x_i\in \mathcal{X}$, let $\ell_i$ and $h_i$ be the vertical and horizontal segments of the lower-envelop of $I_{x_i}$ -- see Figure \ref{fig:query_time} for an illustration. Note that each $h_i$ is in fact a half line in the $x$-axis. 
It is easy to see that the line $L$ intersects $I_{x_i}$ \emph{if and only if} $L$ intersects either $\ell_i$ or $h_i$.

\paragraph{Case 1: reporting intersection with $h_i$s.} 
Given the collection of all $h_i$s, $i\in [1, n]$, in $O(n\log n)$ time, we can build a standard 1D range reporting data structure of size $O(n)$, over the collection of left endpoints $a_i$'s of $h_i$s , $i\in [1,n]\}$, so that given a query point $b$, we can report all points in $\{a_i\}$ to the left of $b$, in $O(\log n + s)$ time where $s$ is the number of such points. 

Now given a query line $L$, let $b_L$ be the intersection between $L$ and the $x$-axis. We use the data strutcure to compute, say $k_1$ number of points from $\{a_i \}$ to the left of $b_L$, in $O(\log n + k_1)$ time. Each such point corresponds to a ray $h_i$ that will intersect $L$. 

\paragraph{Case 2: reporting intersection with $\ell_i$s.}
What remains is to build a data structure to report the set of $\ell_i$s intersecting $L$. 
To this end, note that for each $i\in [1,n]$, the point $a_i$ introduced above is also the bottom endpoint of the vertical segment $\ell_i$; let $t_i$ denote the top endpoint for $\ell_i$. 
Given a query line $L$, we wish to report all $i$'s such that $t_i$ is above $L$ while $a_i$ is below $L$. Again, let $b_L$ denote the intersection of $L$ with the $x$-axis: As the slope of $L$ is positive, if a vertical segment $\ell_i$ intersects $L$, then $a_i$ must lie to the right of $b_L$. 

Now for each $j\in [1, n]$, set 
$$A_j := \{ t_i \mid a_i \ge a_j \}. $$
Given $L$, let $a_r$ be the closest point to $b_L$ with $a_r \ge b_L$. 
Obviously, the line $L$ intersects $\ell_i$ \emph{if and only if} $t_i\in A_r$ and $t_i$ is above $L$. 
Hence we want to perform a halfplane range reporting query among the points in $A_r$. 
To this end, for each $i\in [1, n]$, we use the classic approach of \cite{CGL85} to build a data structure of size $O(|A_i|) =O(n)$ in time $O(|A_i|\log |A_i|) = O(n\log n)$, so that given a line $L$, the set of points from $A_i$ above $L$ can be reported in $O(\log n + s)$ time where $s$ is the number of such points. 
Overall, the total size of all such data structures for all $i\in [1, n]$ is $O(n^2)$ and can be constructed in $O(n^2 \log n)$ time. 
Given $L$, we first identify $a_r$ as described above, and then query for the set of $t_i$s from $A_r$ lying above $L$ in $O(\log n + k_2)$ time, where $k_2$ is the number of such $t_i$s. 

Putting {\sf Case 1} and {\sf Case 2} together, we can report all $k = k_1+k_2$ \newstair{}s $\widehat{\mathcal{I}}_L$ intersecting a query line $L$ of positive slope in time $O(\log n + k)$ time. 

Once we have $\widehat{\mathcal{I}}_L$, for each $I_x \in \widehat{\mathcal{I}}_L$, we use the procedure described at the beginning of this proof to compute the intersection between $L$ and $I_x$ in $O(\log n)$ time for each $I_x$. In total, it takes $O(k \log n)$ to compute all intersections. The total query time is $O(\log n + k + k\log n) = O(k\log n) = O(|B^L|\log n)$ as claimed. 
\end{proof}

\begin{lemma}
Given $I_{\X}$, we can compute the zeroth graded Betti numbers in $O(n^2)$ time. 
\end{lemma}

\begin{proof}
Since the total number of segments of \newbarcodeshort{} is $O(n^2)$ so is the number of corner points. In other words, only $O(n^2)$ grades could potentially have a non-zero $\gamma_i^\X$ or $\beta_i^M$ value, for $i= 0, 1$, or $2$. We can therefore compute graded Betti numbers according to the formula in Section \ref{thm:recover graded Bettis}, by evaluating $\gamma_i^\X$ and $\beta_i^M$ at each of the $O(n^2)$ possible grades. 
\end{proof}

\subsection{Proof of Lemma \ref{lem:recoverStaircode}} \label{app:dec_treegram}

\begin{figure}
\centering
\includegraphics[width=.2\textwidth]{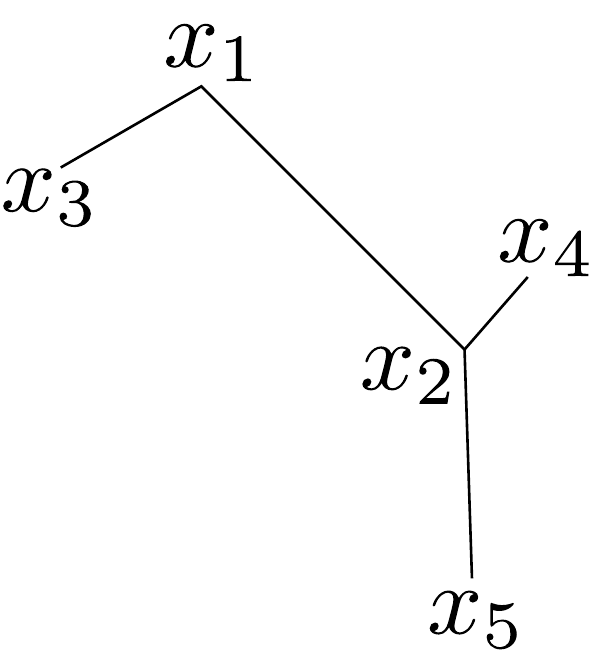}
\hspace{2 cm}
\includegraphics[width=.2\textwidth]{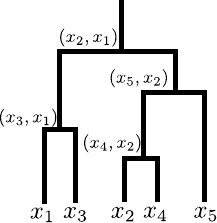}
\caption{The left figure shows a MST of 5 points where node with low(high) index appear early(later). The weight of each edge is the length of the edge (e.g, $w(x_1, x_3) > w(x_2, x_4)$). The right figure shows the corresponding decorated treegram. At each non-leaf node, we only need to store a tuple where the first number stands for the index of the node that is conquered while the second number stands for the index of the node that has not been conquered (eldest) in the subtree. }
\label{fig:decorated_dendrogram}
\end{figure}

\begin{figure}
\centering
\includegraphics[scale=.5]{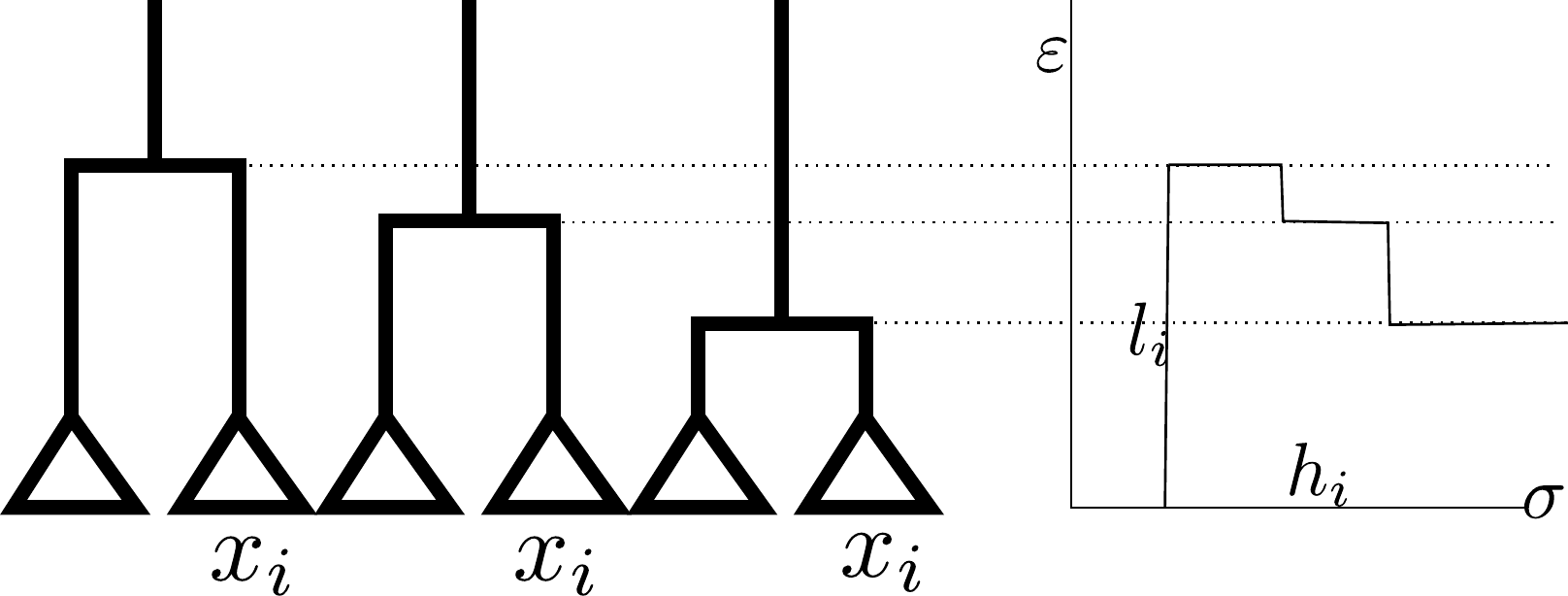}
\caption{Illustration of assembling process to recover $I_{x_i}$. Note that we do not plot the whole treegram at each function value for simplicity. $x_i$ here is a leaf in the right subtree of every treegram.
We will first compute the decorated treegrams, illustrated in Figure \ref{fig:decorated_dendrogram}. From these decorated treegrams, we are able to retrieve $\epsilon$ values of $I_x$ at for each of the $n$  function values $\sigma_1, \ldots, \sigma_n$ with $\sigma_i = f(x_i)$ and thus assemble $I_{x_i}$.} 
\label{fig:assemble}
\end{figure}

We now give a detailed description of the process to recover ER-staircode from the collection of treegrams in $O(n^2)$ time. Recall conqueror is defined in Section \ref{sec:decorated staircode}. When $x'$ is a conqueror of $x$ in $G_i=G_{\sigma_i}$, we also say $x$ is conquered by $x'$ at \emph{height} $u_{X_{\sigma_i}}(x, x')$. To convert treegrams at different function values to \newstair{}, we will decorate the treegrams with some extra information. On the high level, we need to keep the information about the node index conquered at different heights in the treegram, which can be done in linear time by traversing the treegram from bottom to top. 

Specifically, denote the sorted height values of treegram $\theta_i$ at function value $\sigma_i$ as $\mathcal{E}_i =\{\epsilon_1 < \epsilon_2 < ..., \epsilon_{i-1}\}$.
At each non-leaf node of height $\epsilon_j \in \mathcal{E}_i$ in the treegram $\theta_i$, we record a) the index of the node that is conquered at height $\epsilon_j$ and b) index of the single node in subtrees (rooted at height $\epsilon_j$) who has not been conquered yet. b) is needed to update a) of the node at height $\epsilon_{j+1}$ in constant time. Traversing treegrams bottom-up and compute a) and b) for every non-rooted node takes $O(i)=O(n)$ time. An illustration of the idea of decorated treegrams is shown in Figure \ref{fig:decorated_dendrogram}.

After computing $n$ decorated treegrams at $n$ function values, we can recover \newbarcodeshort{} by assembling decorated treegrams in the following way. Without loss of generality, we  state the process to recover single $I_x$ in \newbarcodeshort{}. For every function value $\epsilon_i$, find corresponding $\sigma$ (i.e., $u_{X_{\sigma_i}}(x, x')$) in $\mathcal{E}_i$ at which $x$ is conquered. Repeat this process for all function values will recover $I_x$. Figure \ref{fig:assemble} illustrates the idea.

We restate the Lemma \ref{lem:recoverStaircode} with a proof.

\begin{lemma}
A decorated \newbarcodeshort{} for the finite \filterMSshort{} $(X, d_X, f_X)$ can be computed from the collection of dendrograms $\{ \theta_i , i\in [1, n] \}$ in $O(n^2)$ time. 
\end{lemma}
\begin{proof}
The decoration of every treegrams takes $O(n)$ time and in total $O(n^2)$ for $n$ treegrams. Assembling $I_x$ for each $x\in \mathcal{X}$ takes $O(n)$ time since the complexity of every $I_x$ is $O(n)$ so in total recovering \newbarcodeshort{} takes $O(n^2)$ time.
For the correctness, we prove our process can recover $I_x$ for every $x \in \mathcal{X}$. This holds because  for any $x \in \mathcal{X}$ and $\sigma_i \in f_X$ we can recover $u_{X_{\sigma_i}}(x, x')$ where $x'$ is the conqueror of $x$. 
\end{proof}

\bibliographystyle{plain}
\bibliography{biblio}

\begin{thebibliography}{10}

\bibitem{asashiba2018interval}
Hideto Asashiba, Micka{\"e}l Buchet, Emerson~G Escolar, Ken Nakashima, and
  Michio Yoshiwaki.
\newblock On interval decomposability of 2{D} persistence modules.
\newblock {\em arXiv preprint arXiv:1812.05261}, 2018.

\bibitem{asashiba2019approximation}
Hideto Asashiba, Emerson~G Escolar, Ken Nakashima, and Michio Yoshiwaki.
\newblock On approximation of $2 $ {D} persistence modules by
  interval-decomposables.
\newblock {\em arXiv preprint arXiv:1911.01637}, 2019.

\bibitem{azumaya1950corrections}
Gor{\^o} Azumaya et~al.
\newblock Corrections and supplementaries to my paper concerning
  krull-remak-schmidt's theorem.
\newblock {\em Nagoya Mathematical Journal}, 1:117--124, 1950.

\bibitem{bauer2019cotorsion}
Ulrich Bauer, Magnus~B Botnan, Steffen Oppermann, and Johan Steen.
\newblock Cotorsion torsion triples and the representation theory of filtered
  hierarchical clustering.
\newblock {\em arXiv preprint arXiv:1904.07322}, 2019.

\bibitem{botnan2018algebraic}
Magnus Botnan and Michael Lesnick.
\newblock Algebraic stability of zigzag persistence modules.
\newblock {\em Algebraic \& geometric topology}, 18(6):3133--3204, 2018.

\bibitem{brooksbank2008testing}
Peter~A Brooksbank and Eugene~M Luks.
\newblock Testing isomorphism of modules.
\newblock {\em Journal of Algebra}, 320(11):4020--4029, 2008.

\bibitem{cai2020understanding}
Chen Cai and Yusu Wang.
\newblock Understanding the power of persistence pairing via permutation test.
\newblock {\em arXiv preprint arXiv:2001.06058}, 2020.

\bibitem{campello2013density}
Ricardo~JGB Campello, Davoud Moulavi, and J{\"o}rg Sander.
\newblock Density-based clustering based on hierarchical density estimates.
\newblock In {\em Pacific-Asia conference on knowledge discovery and data
  mining}, pages 160--172. Springer, 2013.

\bibitem{CZ09}
G.~Carlsson and A.~Zomorodian.
\newblock The theory of multidimensional persistence.
\newblock {\em Discrete \& Computational Geometry}, 42(1):71--93, 2009.

\bibitem{carlsson2010characterization}
Gunnar Carlsson and Facundo M{\'e}moli.
\newblock Characterization, stability and convergence of hierarchical
  clustering methods.
\newblock {\em Journal of machine learning research}, 11(Apr):1425--1470, 2010.

\bibitem{carlsson2010multiparameter}
Gunnar Carlsson and Facundo M{\'e}moli.
\newblock Multiparameter hierarchical clustering methods.
\newblock In {\em Classification as a Tool for Research}, pages 63--70.
  Springer, 2010.

\bibitem{carlsson2009theory}
Gunnar Carlsson and Afra Zomorodian.
\newblock The theory of multidimensional persistence.
\newblock {\em Discrete \& Computational Geometry}, 42(1):71--93, 2009.

\bibitem{carriere2019general}
Mathieu Carriere, Fr{\'e}d{\'e}ric Chazal, Yuichi Ike, Th{\'e}o Lacombe, Martin
  Royer, and Yuhei Umeda.
\newblock A general neural network architecture for persistence diagrams and
  graph classification.
\newblock {\em arXiv preprint arXiv:1904.09378}, 2019.

\bibitem{cerri2013betti}
Andrea Cerri, Barbara~Di Fabio, Massimo Ferri, Patrizio Frosini, and Claudia
  Landi.
\newblock Betti numbers in multidimensional persistent homology are stable
  functions.
\newblock {\em Mathematical Methods in the Applied Sciences},
  36(12):1543--1557, 2013.

\bibitem{chacholski2017combinatorial}
Wojciech Chacholski, Martina Scolamiero, and Francesco Vaccarino.
\newblock Combinatorial presentation of multidimensional persistent homology.
\newblock {\em Journal of Pure and Applied Algebra}, 221(5):1055--1075, 2017.

\bibitem{chazal2009gromov}
Fr{\'e}d{\'e}ric Chazal, David Cohen-Steiner, Leonidas~J Guibas, Facundo
  M{\'e}moli, and Steve~Y Oudot.
\newblock Gromov-{H}ausdorff stable signatures for shapes using persistence.
\newblock In {\em Computer Graphics Forum}, volume 28 (5), pages 1393--1403.
  Wiley Online Library, 2009.

\bibitem{chazal2011scalar}
Fr{\'e}d{\'e}ric Chazal, Leonidas~J Guibas, Steve~Y Oudot, and Primoz Skraba.
\newblock Scalar field analysis over point cloud data.
\newblock {\em Discrete \& Computational Geometry}, 46(4):743, 2011.

\bibitem{CGL85}
Bernard Chazelle, Leo~J. Guibas, and D.~T. Lee.
\newblock The power of geometric duality.
\newblock {\em BIT Numerical Mathematics}, 25(1):76--90, Mar 1985.

\bibitem{chin1978algorithms}
Francis Chin and David Houck.
\newblock Algorithms for updating minimal spanning trees.
\newblock {\em Journal of Computer and System Sciences}, 16(3):333--344, 1978.

\bibitem{cohen2007stability}
David Cohen-Steiner, Herbert Edelsbrunner, and John Harer.
\newblock Stability of persistence diagrams.
\newblock {\em Discrete \& Computational Geometry}, 37(1):103--120, 2007.

\bibitem{curry2018fiber}
Justin Curry.
\newblock The fiber of the persistence map for functions on the interval.
\newblock {\em Journal of Applied and Computational Topology}, 2(3-4):301--321,
  2018.

\bibitem{dey2018computing}
Tamal~K Dey and Cheng Xin.
\newblock Computing bottleneck distance for $2 $-d interval decomposable
  modules.
\newblock In {\em Proceedings of the thirty-fourth International Symposium on
  Computational Geometry (SoCG 2018)}, pages 32:1--32:15, 2018.

\bibitem{dey2019generalized}
Tamal~K Dey and Cheng Xin.
\newblock Generalized persistence algorithm for decomposing multi-parameter
  persistence modules.
\newblock {\em arXiv preprint arXiv:1904.03766}, 2019.

\bibitem{edelsbrunner2010computational}
Herbert Edelsbrunner and John Harer.
\newblock {\em Computational topology: an introduction}.
\newblock American Mathematical Soc., 2010.

\bibitem{eisenbud2013commutative}
David Eisenbud.
\newblock {\em Commutative Algebra: with a view toward algebraic geometry},
  volume 150.
\newblock Springer Science \& Business Media, 2013.

\bibitem{harrington2019stratifying}
Heather~A Harrington, Nina Otter, Hal Schenck, and Ulrike Tillmann.
\newblock Stratifying multiparameter persistent homology.
\newblock {\em SIAM Journal on Applied Algebra and Geometry}, 3(3):439--471,
  2019.

\bibitem{Hofer2017Deep}
Christoph Hofer, Roland Kwitt, Marc Niethammer, and Andreas Uhl.
\newblock Deep learning with topological signatures.
\newblock In {\em Advances in Neural Information Processing Systems}, pages
  1634--1644, 2017.

\bibitem{kim2018rank}
Woojin Kim and Facundo M{\'e}moli.
\newblock Generalized persistence diagrams for persistence modules over posets.
\newblock {\em arXiv preprint arXiv:1810.11517}, 2018.

\bibitem{kim2018spatio-temporal}
Woojin Kim and Facundo M{\'e}moli.
\newblock \href{https://doi.org/10.1007/s00454-019-00168-w}{Spatiotemporal
  persistent homology for dynamic metric spaces}.
\newblock {\em Discrete {\&} Computational Geometry}, pages 1--45, 2020.

\bibitem{knudson2007refinement}
Kevin~P. Knudson.
\newblock A refinement of multi-dimensional persistence.
\newblock {\em Homology, Homotopy and Applications}, 10(1):259--281, 2008.

\bibitem{landi2018rank}
Claudia Landi.
\newblock The rank invariant stability via interleavings.
\newblock In {\em Research in Computational Topology}, pages 1--10. Springer,
  2018.

\bibitem{lesnick}
Michael Lesnick.
\newblock The theory of the interleaving distance on multidimensional
  persistence modules.
\newblock {\em Found. Comput. Math.}, 15(3):613--650, June 2015.

\bibitem{lesnick2015interactive}
Michael Lesnick and Matthew Wright.
\newblock Interactive visualization of 2-d persistence modules.
\newblock {\em arXiv preprint arXiv:1512.00180}, 2015.

\bibitem{lesnick2019computing}
Michael Lesnick and Matthew Wright.
\newblock Computing minimal presentations and betti numbers of 2-parameter
  persistent homology.
\newblock {\em arXiv preprint arXiv:1902.05708}, 2019.

\bibitem{martinez2018density}
{\'A}lvaro Mart{\'\i}nez-P{\'e}rez.
\newblock A density-sensitive hierarchical clustering method.
\newblock {\em Journal of Classification}, 35(3):481--510, 2018.

\bibitem{mccleary2019multiparameter}
Alex McCleary and Amit Patel.
\newblock Multiparameter persistence diagrams.
\newblock {\em arXiv preprint arXiv:1905.13220v3}, 2019.

\bibitem{miller2017data}
Ezra Miller.
\newblock Data structures for real multiparameter persistence modules.
\newblock {\em arXiv preprint arXiv:1709.08155}, 2017.

\bibitem{munkres1984elements}
James~R Munkres.
\newblock {\em Elements of algebraic topology}.
\newblock Addison-Wesley Menlo Park, 1984.

\bibitem{patel2018generalized}
Amit Patel.
\newblock Generalized persistence diagrams.
\newblock {\em Journal of Applied and Computational Topology}, 1(3-4):397--419,
  2018.

\bibitem{peeva2010graded}
Irena Peeva.
\newblock {\em Graded syzygies}, volume~14.
\newblock Springer Science \& Business Media, 2010.

\bibitem{smith2016hierarchical}
Zane Smith, Samir Chowdhury, and Facundo M{\'e}moli.
\newblock Hierarchical representations of network data with optimal distortion
  bounds.
\newblock In {\em 2016 50th Asilomar Conference on Signals, Systems and
  Computers}, pages 1834--1838. IEEE, 2016.

\bibitem{vipond2020multiparameter}
Oliver Vipond.
\newblock Multiparameter persistence landscapes.
\newblock {\em Journal of Machine Learning Research}, 21(61):1--38, 2020.

\bibitem{ZW19}
Qi~Zhao and Yusu Wang.
\newblock Learning metrics for persistence-based summaries and applications for
  graph classification.
\newblock In {\em 33rd Annu. Conf. Neural Inf. Processing Systems (NeuRIPS)},
  2019.
\newblock to appear.

\bibitem{zomorodian2005computing}
Afra Zomorodian and Gunnar Carlsson.
\newblock Computing persistent homology.
\newblock {\em Discrete \& Computational Geometry}, 33(2):249--274, 2005.

\end{thebibliography}
\end{document}